\documentclass[11pt,a4paper]{article}
\usepackage{authblk}
\usepackage[utf8]{inputenc}
\usepackage[T1]{fontenc}\usepackage{pdfsync}
\usepackage{amsmath}
\usepackage{comment}
\usepackage{calrsfs}
\usepackage{amsfonts}
\usepackage{amssymb}
\usepackage{amsthm}
\usepackage{xcolor}
\usepackage{bbm,graphicx}
\usepackage{subcaption}

\usepackage[left=2cm,right=2cm,top=2cm,bottom=2.5cm]{geometry}
\newtheorem{thm}{THEOREM}

\newtheorem{cor}[thm]{COROLLARY}
\newtheorem{lem}[thm]{Lemma}
\newtheorem{definition}[thm]{DEFINITION}
\newtheorem{example}[thm]{EXAMPLE}
\newtheorem{proposition}[thm]{PROPOSITION}
\newtheorem{remark}[thm]{REMARK}
\def\beXa{\begin{example}} \def\eeXa{\end{example}}
\def\eeD{\end{definition}} \def\beD{\begin{definition}}
\def\beR{\begin{remark}} \def\eeR{\end{remark}}
\def\beL{\begin{lem}} \def\eeL{\end{lem}}
\def\beP{\begin{proposition}} \def\eeP{\end{proposition}}
\def\beC{\begin{cor}} \def\beT{\begin{thm}}
  \def\eeT{\end{thm}}
\def\eeC{\end{cor}}

\providecommand{\norm}[1]{\left\lVert#1\right\rVert}
\providecommand{\abs}[1]{\left\lvert#1\right\rvert}
\providecommand{\pr}[1]{\left(#1\right)} 
\providecommand{\pp}[1]{\left[#1\right]} 
\providecommand{\set}[1]{\left\lbrace#1\right\rbrace} 
\providecommand{\scal}[1]{\left\langle#1\right\rangle}

\providecommand{\keywords}[1]{\textbf{\textit{Keywords:  }} #1}

\title{The stochastic fast logarithmic equation in $\mathbb{R}^{d}$ with
multiplicative Stratonovich noise}
\author[1,2]{Ioana Ciotir}
\author[2]{Reika Fukuizumi}
\author[3,4,*]{Dan Goreac} 
\affil[1]{Normandie University, INSA de Rouen Normandie,
LMI (EA 3226 - FR CNRS 3335), 76000 Rouen, France,\textit{email:ioana.ciotir@insa-rouen.fr}} 
\affil[2]{Research Center for Pure and Applied Mathematics,
Graduate School of Information Sciences,
Tohoku University, Japan,  \textit{email:fukuizumi@math.is.tohoku.ac.jp}}
\affil[3]{School of Mathematics and Statistics, Shandong University, Weihai, Weihai 264209, PR China} 
\affil[4]{LAMA, Univ Gustave Eiffel, UPEM, Univ Paris Est Creteil, CNRS, F-77447 Marne-la-Vallée, France, \textit{email:dan.goreac@univ-eiffel.fr}}
\affil[*]{Corresponding author, email: dan.goreac@univ-eiffel.fr, dan.goreac@u-pem.fr}

\begin{document}
\maketitle

\begin{abstract}
This paper is concerned with the existence and uniqueness of the solution for the stochastic fast logarithmic equation with Stratonovich multiplicative noise in $\mathbb{R}^{d}$ for $d\geqslant 3$.  It provides an answer to a critical case (morally speaking, corresponding to the porous media operator $\Delta X^m $ for $ m=0$) left as an open problem in the paper Barbu-Röckner-Russo \cite{BARBU20151024}.
We face several technical difficulties related both to the degeneracy properties of the logarithm and to the fact that the problem is treated in an unbounded domain.  Firstly,  the order in which the approximations are considered is very important and different from previous methods.  Secondly,  the energy estimates (see eq.  \eqref{S3Eq4}) needed in the last step can only be achieved with a well-chosen Stratonovich-type rectification of the noise.\\

\noindent \keywords{stochastic fast logarithmic equation; SPDE in unbounded domains; maximal monotone operators; Yosida approximation; multiplicative Stratonovich noise.}\\

 \noindent \textbf{\textit{2020 MSC:} 76S05, \and 60H15, \and 35K55.}

\end{abstract}

\section{Introduction}

Let us consider a nonlinear diffusion process of the following form 
\begin{equation}
dX\left( t\right) =\Delta \ln \left( X\left( t\right)\right) dt
\label{proces}
\end{equation}
where $X\left( t,\xi \right) $ is the positive density for the time - space
coordinates $\left( t,\xi \right) .$ This equation describes the process that
has been observed by experiments when using Wisconsin toroidal octupole
plasma containment device (see \cite{ehrhardt}). Kamimura and Dawson
predicted in \cite{Kamimura}\ this time evolution for the cross-field conservative
diffusion of plasma including mirror effects.

The same equation is relevant for the expansion of a thermalized electron cloud
and arises in studies of the central limit approximation to Carleman's
model of the Boltzmann equation (see \cite{carleman} and \cite{Kurtz}) as well. The
asymptotic behaviour of this equation has been studied in \cite{berryman}. \\

Most of
the natural phenomena exhibit some form of variability which cannot be captured by using purely
deterministic approaches. More accurately, natural systems could be
represented as stochastically perturbed models and the deterministic description can be
considered as the subset of the pertinent stochastic models.

The purpose of this paper is to analyse such equations within the framework
of stochastic evolution equations with multiplicative noise, where the dynamics (\ref%
{proces}) is the underlying motivating example. 

Let us now introduce the
suitable framework for this problem. We consider the Stratonovich stochastic differential equation (in some sense to be made precise later on) on the Euclidean space $\mathbb{R}%
^{d}$, for some $d\geq 3$, of the form 
\begin{equation}
\left\{ 
\begin{array}{ll}
dX_{t}-\Delta \ln X_{t}dt=X_t\circ dW_{t}, & \quad (t,\xi) \in \left[ 0,T\right]
\times \mathbb{R}^{d}, \\ 
X_{0}=x.& %
\end{array}
\right.  \label{equ}
\end{equation}
The unknown $X_t=X(t,\xi)$ is a real-valued random field on 
a standard complete, right-continuous probability basis $
\left\{ \Omega ,\mathcal{F},\left\{ \mathcal{F}_{t}\right \} _{t},\mathbb{P}
\right\} $, and $W_{t}$ is related to a $\mathcal{H}^{-1}$-valued 
$Q$ Wiener process associated with the filtration 
$\{\mathcal{F}_t\}_{t\ge0}$, where $Q$ is a non negative trace class operator on $\mathcal{H}^{-1}$. For more details  and the properties of the space 
$\mathcal{H}^{-1}$ see the next section and  Section 6.2 of \cite{BDPR_2016}.

The stochastic theory of nonlinear equations has been, recently, intensively
studied for drift coefficients of the form $-\Delta \Psi $, where $\Psi :\mathbb{\
R\rightarrow R}$ defined by $\Psi \left( r\right) =r^{m}$ is a maximal
monotone operator with additive and multiplicative noise.

In the case $m>1,$ the corresponding equation describes the slow diffusions
(dynamics of fluids in porous media) and their existence, uniqueness and
positivity and behavior of the solution have already been studied in \cite%
{positivity}, \cite{strong}, \cite{BG2}, \cite{strong2}, \cite{RRW_2007} for the
stochastic case. For the deterministic case see \cite{ARON} and \cite%
{VASQUEZ}.

The case $m\in \left( 0,1\right) $ is relevant in the mathematical modeling
of dynamics of an ideal gas in a porous media. For the self organized
criticality case see \cite{CCGS}, \cite{criticality}, \cite{BG1}. Finite
time extinction is studied in 3 dimensions for $m\in \left[ \frac{1}{5}
,1\right) $ in \cite{BDPR3}. See also \cite{BDPR12}.

 The stochastic counterpart was studied, for $m \in (-1,0)$ and multiplicative Itô noise structure in \cite{eusuperfast}.

For the case $m\leq -1,$ it has been proved that, even in the deterministic
case, there is no solution with finite mass (see \cite{VASQUEZ}).

The case $\Psi \left( r\right) =\log r$ was studied for a multiplicative
noise in a bounded domain in \cite{LN}. Note that for positive solutions, it
can be seen as morally corresponding to the situations $m=0$,  in   $ \mathit{div}
~(r^{m-1}\nabla r)=~\frac{1}{m} \Delta \left(  r^m \right) $ since $\mathit{div}
~(r^{-1}\nabla r)=~\Delta \left( \ln r\right) .$

Concerning the stochastic porous media equation in an unbounded domain, the
only known result which is known is in \cite{BARBU20151024} and treats the case of slow
diffusion with a multiplicative It\^{o} noise.

Our work which treats the fast logarithmic diffusion in an unbounded domain has several technical difficulties 
which will be treated by using several specific approximation. 

More precisely, the first main set of difficulties comes from the properties of the logarithm. We have a problem due to the fact that zero does not belong to $D(\ln)$ and we can not assume that $D(\ln)=\textbf{R}$. Another problem which is specific to the logarithm diffusion is the fact that we can not assume any polynomial growth hypothesis, nor the strong monotonicity assumption. All those technical difficulties impose the choice of a particular form of the first approximation in $\lambda$ and the use of a Stratonovich multiplicative noise. 

The second main set of difficulties comes from the unboundedness of the domain which implies, among other problems, the impossibility to use the Poincaré inequality. This technical problem impose the use of a second approximation in $\nu$. 

Finally a third approximation in $\epsilon$ is necessary to get some estimates in appropriated spaces. 

One needs to notice also that the order of passing to the limit for the three approximations is rigorously chosen. 
More precisely one needs to pass to the limit in $\nu$ before $\lambda$ in order to avoid that $\Psi_\lambda (0)$ converges to $\infty$. 
This is an important technical difficulty with respect to the case of a slow porous media diffusion.

The organization of the paper is the following. After an introduction, in the second section we have some notations and the technical setting of the problem. The third section is concerned with the definition of the solution and the existence and uniqueness result. The fourth section gives the proof of the main results in several steps corresponding to the approximations presented above. Finally we have an appendix with some technical points.

For reader's convenience we shall recall some basic notions and settings in the second section.

\section{Notations and setting}
\subsection{Functional Spaces}
Throughout the paper we are going to adopt the following notations.
\begin{enumerate}
\item the underlying space will be a $d\geq 3$-dimensional Euclidean space $\mathbb{R}^d$;
\item the fundamental functional space is $\mathbb{L}^2\pr{\mathbb{R}^d}=\mathbb{L}^2\pr{\mathbb{R}^d;\mathbb{R}}$ of real-valued Lebesgue-square integrable functions. Its norm is $\norm{\cdot}_{\mathbb{L}^2\pr{\mathbb{R}^d}}$;
\item in general, for $1<p<\infty$, we let $\mathbb{L}^p\pr{\mathbb{R}^d;\mathbb{R}^{d'}}$ stand for the space of $\mathbb{R}^{d'}$-valued Lebesgue-$p$-power integrable functions. Its norm is denoted by $\norm{\cdot}_{\mathbb{L}^p\pr{\mathbb{R}^d;\mathbb{R}^{d'}}}$.  Whenever $d'=1$, we designate the space by $\mathbb{L}^p\pr{\mathbb{R}^d}$ the corresponding space of real-valued functions;
\item we will often drop completely the dependency on the underlying Euclidean spaces e.g.  $\mathbb{R}^d$;
\item the norms for $\mathbb{L}^p$-spaces will be sometimes shortened to $p$ i.e. $\norm{\cdot}_p$. This will also apply to the duality product $\scal{\cdot,\cdot}_{p,q}$ when $\frac{1}{p}+\frac{1}{q}=1$ and the Hilbert product $ \scal{\cdot,\cdot}_{2}=\scal{\cdot,\cdot}_{2,2}$;
\item The space $H^1\pr{\mathbb{R}^d}=H^1_0\pr{\mathbb{R}^d}$ is the inhomogeneous Sobolev space on $\mathbb{R}^d$ (functions belonging, together with their first-order partial derivatives, to $\mathbb{L}^2\pr{\mathbb{R}^d}$).  Its norm is \[\norm{u}_{H^1\pr{\mathbb{R}^d}}:=\norm{\pr{u,\nabla u}}_{\mathbb{L}^2\pr{\mathbb{R}^d;\mathbb{R}^{1+d}}}=\pr{\int_{\mathbb{R}^d}\pr{u^2(\xi)+\abs{\nabla u(\xi)}^2}d\xi}^{\frac{1}{2}}.\]
\item the dual space of $H^1$ (with pivot space $\mathbb{L}^2\pr{\mathbb{R}^d}$) is denoted by $H^{-1}\pr{\mathbb{R}^d}$ (or simply $H^{-1}$);
\item the space $\mathcal{H}^s\pr{\mathbb{R}^d}$ for $s\in\mathbb{R}$ is the homogeneous Sobolev space of (real-valued) tempered distributions $u$ over $\mathbb{R}^d$ having an $\mathbb{L}^1_{loc}\pr{\mathbb{R}^d}$ Fourier distribution $\hat{u}$ and such that\[\norm{u}_s^2=\norm{u}_{\mathcal{H}^s\pr{\mathbb{R}^d}}^2:=\int_{\mathbb{R}^d}\abs{\xi}^{2s}\abs{\hat{u}(\xi)}^2d\xi<\infty.\]

\item For any $s \in \mathbb{R}$, $C_{\mathbb{P}}([0,T]; \mathcal{H}^s)$ denotes the space of all $\mathcal{H}^s$-valued  $(\mathcal{F}_t)_t$ progressively measurable processes $X: \Omega \times  [0,T] \to \mathcal{H}^s$ such that 
$$ \mathbb{E} \int_0^T \|X(t)\|^2_{\mathcal{H}^s} dt < +\infty,$$
and for all compacts $\mathcal{O}$,  the realization of $X$ on 
$\mathcal{O}$ has a continuous modification in \\
$C([0,T]; L^2(\Omega, \mathcal{H}^s(\mathcal{O})))$.   
\end{enumerate}
\begin{remark}\label{RemIncl}
\begin{enumerate}
\item The space $H^1$ is associated with the space $\mathcal{F}_e$ defined on \cite[page 129]{RRW_2007} and associated to the operator $L=\bigtriangleup$ with the corresponding domain in $\mathbb{L}^2\pr{\mathbb{R}^d}$.  The accompanying quadratic form $\mathcal{E}(u,v):=\scal{\sqrt{-L}u,\sqrt{-L}v}_{\mathbb{L}^2\pr{\mathbb{R}^d}}$ renders an extended transient Dirichlet space structure such that one has \begin{align}\label{Gelfand}
V:=\mathbb{L}^2\pr{\mathbb{R}^d}\subset H^{-1}\subset V^*\textnormal{ continuously and densely.}
\end{align}
(cf.  \cite[page 131]{RRW_2007}, see also the explicit example  \cite[page 129, paragraph preceding Eq.  (3.1)]{RRW_2007}.)
\item The space $\mathcal{H}^s\pr{\mathbb{R}^d}$ is a Hilbert space provided $s<\frac{d}{2}$(cf.  \cite[Prop. 1.34]{BCD_2011}).
\item Provided that $\abs{s}<\frac{d}{2}$, the spaces $\mathcal{H}^s\pr{\mathbb{R}^d}$ and $\mathcal{H}^{-s}\pr{\mathbb{R}^d}$ are dual (cf. \cite[Prop. 1.36]{BCD_2011}). 
\begin{align}\label{DualHsH-s}
\scal{u,v}_{\pr{\mathcal{H}^{s}\pr{\mathbb{R}^d},\mathcal{H}^{-s}\pr{\mathbb{R}^d}}}=\int_{\mathbb{R}^d}u(\xi)v(\xi)d\xi,\ \forall \pr{u,v}\in\mathcal{H}^{s}\pr{\mathbb{R}^d}\times\mathcal{H}^{-s}\pr{\mathbb{R}^d}.
\end{align}
\item The following embeddings hold true (for $0\leq s<\frac{d}{2}$, cf. \cite[Theorem 1.38, Corollary 1.39]{BCD_2011}).
\begin{align}
\label{EmbedCalH}
\mathcal{H}^s\pr{\mathbb{R}^d}\subset \mathbb{L}^{\frac{2d}{d-2s}}\pr{\mathbb{R}^d},\ \mathbb{L}^{\frac{2d}{d+2s}}\pr{\mathbb{R}^d}\subset\mathcal{H}^{-s}\pr{\mathbb{R}^d}\textnormal{ continuously.}
\end{align}
Embedding constants will be denoted by $C_{\mathcal{H}^s\subset\mathbb{L}^{\frac{2d}{d-2s}}}$.  They are assumed to be at least $1$ such that \[\norm{u}_{\mathbb{L}^{\frac{2d}{d-2s}}\pr{\mathbb{R}^d}}\leq C_{\mathcal{H}^s\subset\mathbb{L}^{\frac{2d}{d-2s}}}\norm{u}_{\mathcal{H}^s\pr{\mathbb{R}^d}}.\]An explicit expression of such constants can be found in \cite[Theorem 1.38]{BCD_2011} (last line of the proof).
\item In the case when $s=1$,  a more convenient expression of $\norm{\cdot}_{\mathcal{H}^{1}\pr{\mathbb{R}^d}}$ is given by \[\norm{u}_{\mathcal{H}^{1}\pr{\mathbb{R}^d}}=\norm{\nabla u}_{\mathbb{L}^2\pr{\mathbb{R}^d;\mathbb{R}^d}}.\]
\end{enumerate}
\end{remark}
\subsection{Basis and Wiener Process}
One recalls that $\mathcal{S}_0\pr{\mathbb{R}^d}$, the space of smooth functions (belonging to the Schwartz space) such that the Fourier transform of which vanishes near origin. It is known that $\mathcal{S}_0 (\mathbb{R}^d)$ is dense in $\mathcal{H}^{-1}$ (see, for example, \cite[Prop. 1.35]{BCD_2011}). As a consequence, one is able to pick an orthonormal basis for $\mathcal{H}^{-1}$ whose elements $e_k\in C^1\pr{\mathbb{R}^d}\cap\mathcal{H}^{-1}$ 
with its first order derivatives in  $\mathbb{L}^d\pr{\mathbb{R}^d}$ (to see the required integrability, just recall the linear growth on such smooth functions).
In particular, for every $k\geq 1$, \[\mu'_k:=\norm{e_k}^2_{\mathbb{L}^\infty\pr{\mathbb{R}^d}}+\norm{\nabla e_k}^2_{\mathbb{L}^d\pr{\mathbb{R}^d;\mathbb{R}^d}}+1<\infty.\]
Next, let us consider a sequence $\pr{\mu_k}_{k\geq 1}\subset\mathbb{R}_+^{\mathbb{N}}$ such that 
\begin{align}\label{Cond_mu}
\sum_{1\leq k}\mu_k\mu_k'=\sum_{1\leq k}\mu_k\pr{\norm{e_k}^2_{\infty}+\norm{\nabla e_k}_{d}^2+1}<\infty.
\end{align}
To this sequence $\mu_k$, we associate an $\mathcal{H}^{-1}$-valued $Q$  Wiener process, i.e., 
$Q$ is a non-negative trace-class operator such that $Qe_k=\mu_ke_k$ on $\mathcal{H}^{-1}$. 
The formal representation yields
$W^Q(t)$, 
\[
W^Q(t)  =
\sum_{k\geq 1}\sqrt{\mu_k}\beta_k(t)e_k,\ t\in\pp{0,T},\]for some sequence of independent real-valued Brownian motions $\{\beta_k\}_{k}$ defined on a common space $\pr{\Omega,\mathcal{F},\mathbb{P}}$. 
For more details see Remark 3.3 from \cite{BARBU20151024} or Section 2.3 from 
\cite{Prevot-Rockner}.


\subsubsection{On $H^{-1}\pr{\mathbb{R}^d}$ With Standard and Modified Norms}

\begin{enumerate}

\item The space $H^{-1}\pr{\mathbb{R}^d}$ is usually endowed with the norm\[\norm{u}_{H^{-1}\pr{\mathbb{R}^d}}^2:=\int_{\mathbb{R}^d}\pr{1+\abs{\xi}^2}^{-1}\abs{\hat{u}(\xi)}^2\ d\xi,\] or, again, $\norm{u}_{H^{-1}\pr{\mathbb{R}^d}}^2=\norm{\pr{\mathbb{I}-\bigtriangleup}^{-\frac{1}{2}}u}^2_{\mathbb{L}^2}$.  Equivalently, for $\nu>0$, one can consider \[\norm{u}_{H^{-1}_\nu\pr{\mathbb{R}^d}}^2:=\norm{\pr{\nu\mathbb{I}-\bigtriangleup}^{-\frac{1}{2}}u}^2_{\mathbb{L}^2};\]

\item \textbf{The process $W$ as a Brownian motion on $H^{-1}$ and Itô differentials} 
\begin{itemize}
\item One considers 
$J:Q^{\frac{1}{2}}\pr{\mathcal{H}^{-1}}\rightarrow H^{-1}$ given by $J\pr{Q^{\frac{1}{2}}u}:=\sum_{k\geq 1}\sqrt{\mu_k}\scal{u,e_k}_{\mathcal{H}^{-1}}e_k$.  One easily notes this to be a Hilbert-Schmidt embedding.  \begin{align*}
\sum_{k\geq 1}\norm{J(Q^{\frac{1}{2}}e_k)}_{H^{-1}}^2=\sum_{k\geq 1}\mu_k\norm{e_k}_{H^{-1}}^2
\leq\sum_{k\geq 1}\mu_k\norm{e_k}_{\mathcal{H}^{-1}}^2<\infty,
\end{align*}
the last but one inequality being a simple consequence of the continuous embedding $\mathcal{H}^{-1}\subset H^{-1}$ and equality of pair product.

\item The noise coefficient is regarded for given $x \in H^{-1}$ as $\sigma(x):Q^{\frac{1}{2}}\mathcal{H}^{-1}\rightarrow H^{-1}$ defined by $\sigma(x)\pr{Q^{\frac{1}{2}}u}:=\sum_{k\geq 1}\sqrt{\mu_k}\scal{e_k,u}_{\mathcal{H}^{-1}}e_kx$, 
which is well-defined, since, whenever $\phi\in C^\infty_0\pr{\mathbb{R}^d}$ with $\norm{\phi}_{H^1}\leq 1$, a simple computation then yields 
\begin{align*}\norm{e_kx}_{H^{-1}}=\sup\ \set{\scal{x,e_k\phi}_{\pr{H^{-1},H^1}}:\ \norm{\phi}_{H^1}\leq 1}\leq \norm{x}_{H^{-1}}\sup_{\norm{\phi}_{H^1}\leq 1} \norm{e_k\phi}_{H^{1}}.
\end{align*}
Here, 
\begin{align*}
\norm{e_k\phi}_{H^{1}}^2&=\norm{e_k\phi}_{\mathbb{L}^2}^2+\norm{\nabla \pr{e_k\phi}}_{\mathbb{L}^2}^2\leq \pr{\norm{\phi}_{2}^2\norm{e_k}_\infty^2+\norm{\nabla e_k}_{\mathbb{L}^d}^2\norm{\phi}_{\mathbb{L}^{\frac{2d}{d-2}}}^2+\norm{e_k}^2_{\infty}\norm{\nabla\phi}_2^2}\\
&\leq \norm{e_k}^2_{\infty}+\norm{\nabla e_k}_{\mathbb{L}^d}^2\norm{\phi}_{\mathbb{L}^{\frac{2d}{d-2}}}^2\leq \norm{e_k}^2_{\infty}+C_{H^1\subset\mathbb{L}^{\frac{2d}{d-2}}}^{2}\norm{\nabla e_k}_d^2,
\end{align*}(the latter inequality follows from the continuous embeddings $H^1\subset\mathcal{H}^1\subset \mathbb{L}^{\frac{2d}{d-2}}$) which amounts to 
\[\norm{e_kx}_{H^{-1}}^2\leq \norm{x}_{H^{-1}}^2\pr{ \norm{e_k}^2_{\infty}+C^2_{H^1\subset\mathbb{L}^{\frac{2d}{d-2}}}\norm{\nabla e_k}_d^2}.\] 
As a consequence, 
$\sigma(x)\in\mathcal{L}\pr{Q^{\frac{1}{2}}\mathcal{H}^{-1};H^{-1}}$. Moreover, 
\begin{equation}\label{normSigmaH-1_v0}\begin{split}\norm{\sigma(x)}^2_{\mathcal{L}_2\pr{Q^{\frac{1}{2}}\mathcal{H}^{-1};H^{-1}}}&=\sum_{k\geq 1}\norm{\sigma(x)\pr{Q^{\frac{1}{2}}e_k}}^2_{H^{-1}}=\sum_{k\geq 1}\mu_k\norm{e_kx}_{H^{-1}}^2\\&\leq \norm{x}^2_{H^{-1}}\sum_{k\geq 1}\mu_k\pr{\norm{e_k}^2_{\infty}+C^2_{H^1\subset\mathbb{L}^{\frac{2d}{d-2}}}\norm{\nabla e_k}_d^2}.\end{split}\end{equation}
\end{itemize}

\item \textbf{On $H^1_\nu$.}\\

The mapping $\left(0,1\right]\ni \nu\mapsto \norm{\cdot}_{H_\nu^{-1}\pr{\mathbb{R}^d}}$ is non-increasing while $\left(0,1\right]\ni \nu\mapsto \norm{\cdot}_{H_\nu^{1}\pr{\mathbb{R}^d}}$ is non-decreasing. As a by-product,  the constant $C_{H^1_\nu\subset\mathbb{L}^{\frac{2d}{d-2}}}$ appearing in \eqref{normSigmaH-1_v0} (in the case $\nu=1$) can be chosen independently of $\nu\in(0,1]$. Formally if we regard $\mathcal{H}^1$ as the case of  $\nu=0$ (which corresponds to the embedding $\mathcal{H}^{1}\subset \mathbb{L}^{\frac{2d}{d-2}}$), for $x \in H^{-1}_{\nu}$ 
\begin{equation}\label{normSigmaH}\begin{split}\norm{\sigma(x)}^2_{\mathcal{L}_2\pr{Q^{\frac{1}{2}}\mathcal{H}^{-1};H^{-1}_\nu}}&\leq \norm{x}^2_{H^{-1}_\nu}\sum_{k\geq 1}\mu_k\pr{\norm{e_k}^2_{\infty}+C^2_{\mathcal{H}^1\subset\mathbb{L}^{\frac{2d}{d-2}}}\norm{\nabla e_k}_d^2}.\end{split}\end{equation}

\item The same argument can be developed on $\mathcal{H}^{-1}$ leading to,  for $x \in \mathcal{H}^{-1}$ \begin{equation}\label{normSigmaCalH}\begin{split}\norm{\sigma(x)}^2_{\mathcal{L}_2\pr{Q^{\frac{1}{2}}\mathcal{H}^{-1};\mathcal{H}^{-1}}}&\leq \norm{x}^2_{\mathcal{H}^{-1}}\sum_{k\geq 1}\mu_k\pr{\norm{e_k}^2_{\infty}+C^2_{\mathcal{H}^1\subset\mathbb{L}^{\frac{2d}{d-2}}}\norm{\nabla e_k}_d^2}.\end{split}\end{equation}

\item Finally, the estimates \eqref{normSigmaH} and \eqref{normSigmaCalH} are useful for Itô's isometry yielding 
\begin{equation}\label{ItoIsom}\begin{split}
&\mathbb{E}\pp{\norm{\int_0^t\sigma\pr{X(s)}dW^Q(s)}^2_{\mathbb{H}}}=\sum_{k\geq 1}\mu_k\mathbb{E}\pp{\int_0^t\norm{X(s)e_k}^2_{\mathbb{H}}ds}\\&\leq \pr{\sum_{k\geq 1}\mu_k\pr{\norm{e_k}^2_{\infty}+C^2_{\mathcal{H}^1\subset\mathbb{L}^{\frac{2d}{d-2}}}\norm{\nabla e_k}_d^2}}\mathbb{E}\pp{\int_0^t\norm{X(s)}_{\mathbb{H}}^2ds},
\end{split}\end{equation}
where $\mathbb{H}\in\set{\mathcal{H}^{-1},\ H^{-1}, \ H^{-1}_\nu,\nu\in(0,1)}$.
\end{enumerate}
\subsubsection{On $\mathbb{L}^2\pr{\mathbb{R}^d}$}
The operator $J$ considered before provides one with an element $J\pr{Q^{\frac{1}{2}}u}=\underset{k\geq 1}{\sum}\sqrt{\mu_k}\scal{u,e_k}_{\mathcal{H}^{-1}}e_k\in \mathbb{L}^2\pr{\mathbb{R}^d}$. \\
By noting that,  for a given $x \in \mathbb{L}^2$,
\begin{equation}\label{normL2}
\norm{e_kx}_{\mathbb{L}^2}\leq \norm{x}_{\mathbb{L}^2}\norm{e_k}_{\infty},
\end{equation}it follows, as before, that
\begin{equation}
\norm{\sigma(x)}^2_{\mathcal{L}_2\pr{Q^{\frac{1}{2}}\mathcal{H}^{-1};\mathbb{L}^2}}\leq \sum_{k\geq 1}\mu_k\pr{\norm{e_k}^2_{\infty}+C^2_{\mathcal{H}^1\subset\mathbb{L}^{\frac{2d}{d-2}}}\norm{\nabla e_k}_d^2}\norm{x}_{\mathbb{L}^2}^2.
\end{equation}

\subsection{Some Stratonovich Considerations}
Let us consider, on one of the spaces $\mathbb{H}$ ($\mathbb{H}\in\set{\mathcal{H}^{-1},\ H^{-1}, \ H_\nu^{-1},\nu\in(0,1),\ \mathbb{L}^2}$), the equation \[dX(t)=f(X(t))dt+\sigma(X(t))dW^Q(t)=\sigma(X(t))\circ dW^Q(t),\ t\in\pp{0,T},\]with the Itô and Stratonovich differentials.  Heuristically speaking, according to Itô's formula (e.g.  \cite[Theorem 4.2.5]{Prevot-Rockner}) for $\norm{\cdot}^2_{\mathbb{H}}$,  one gets
\[
\begin{split}
\norm{X(t)}^2_{\mathbb{H}}=&\norm{X(0)}^2_{\mathbb{H}}+2\int_0^t\scal{X(s),\sigma(X(s))dW^Q(s)}_{\mathbb{H}}\\
& +\int_0^t\pr{2\scal{f(X(s)),X(s)}_{\mathbb{H}}+\norm{\sigma(X(s)}^2_{\mathcal{L}_2\pr{Q^{\frac{1}{2}}\mathcal{H}^{-1};\mathbb{H}}}}ds.
\end{split}
\]
In the spirit of Stratonovich differentials (e.g.  \cite{Ito_1975}), one would expect \[d\norm{X(t)}^2_{\mathbb{H}}=2\scal{X(t),\sigma(X(t))\circ dW^Q(t)}_{\mathbb{H}}.\]As such, a simple glance at the first line of \eqref{ItoIsom} leads one to look into the application $\mathbb{H}\ni x\mapsto \pr{\sqrt{\mu_k}xe_k}_{k\geq 1}\in {l}^2\pr{\mathbb{H}}$. It belongs to $\mathcal{L}\pr{\mathbb{H};{l}^2(\mathbb{H})}$ and, as a by-product,  the application \[\mathbb{H}\times \mathbb{H}\ni(x,y)\mapsto \tilde{f}(x,y)=\sum_{k\geq 1}\mu_k\scal{xe_k,ye_k}_{\mathbb{H}}\in\mathbb{R}\textnormal{ is a bilinear (bounded) form.}\]As before,\[\tilde{f}(x,y)\leq  \pr{\sum_{k\geq 1}\mu_k\pr{\norm{e_k}^2_{\infty}+C^2_{\mathcal{H}^1\subset\mathbb{L}^{\frac{2d}{d-2}}}\norm{\nabla e_k}_d^2}} \norm{x}_{\mathbb{H}}\norm{y}_{\mathbb{H}}.\]One can then consider 
$\sigma\otimes\sigma:\mathbb{H}\rightarrow\mathbb{H}$ defined,  via Riesz identification \begin{align}\label{SigSig}\scal{\pr{\sigma\otimes\sigma}(x),y}_{\mathbb{H}}=\tilde{f}(x,y)=\scal{\sigma(x),\sigma(y)}_{\mathcal{L}_2\pr{Q^{\frac{1}{2}}\mathcal{H}^{-1};\mathbb{H}}}. \end{align}and note that 
\begin{align}\label{normSigSig}\norm{\sigma\otimes\sigma}_{\mathcal{L}\pr{\mathbb{H};\mathbb{H}}}=\norm{\sigma}^2_{\mathcal{L}\pr{\mathbb{H};\mathcal{L}_2\pr{Q^{\frac{1}{2}}\mathcal{H}^{-1};\mathbb{H}}}}.\end{align}
It now appears obvious that one should set (in the Stratonovich sense) 
\begin{align}\label{strat}
X(t)\circ dW(t):=\sigma\pr{X(t)}\circ dW^Q(t):=X(t)dW (t)+\frac{1}{2}\pr{\sigma\otimes\sigma}\pr{X(t)}dt.
\end{align}
This is coherent with the Itô case from Remark 3.3 from \cite{BARBU20151024} or Section 2.3 from 
\cite{Prevot-Rockner}.

\subsection{The $\mathbb{L}^2\pr{\mathbb{R}^d}$ construction}
For the aim of this paper, we make the previous identification in $\mathbb{H}:=\mathbb{L}^2\pr{\mathbb{R}^d}$. Further elements on the construction on $\mathcal{H}^{-1}$ and the compatibility with the spaces $H^{-1}_\nu$ are given in the Appendix \ref{AppStrat} but they will not be used for our setting.
In other words, we define \begin{equation}
\label{SigSigDef}
\pr{\sigma\otimes\sigma}(x):=\sum_{k\geq 1}\mu_ke_k^2x.
\end{equation}
It is simple to check that \begin{align}
\label{PropSigSig}\begin{cases}
\sigma\otimes\sigma\in\mathcal{L}\pr{\mathbb{L}^2;\mathbb{L}^2},\ \norm{\sigma\otimes\sigma}_{\mathcal{L}\pr{\mathbb{L}^2;\mathbb{L}^2}}\leq \underset{k\geq 1}{\sum}\mu_k\norm{e_k}_\infty^2;\\
\scal{\pr{\sigma\otimes\sigma}(x),x}_{\mathbb{L}^2}=\norm{\sigma(x)}_{\mathcal{L}_2\pr{Q^{\frac{1}{2}}\mathcal{H}^{-1};\mathbb{L}^2}}^2;\\
\norm{\sigma\otimes\sigma}_{\mathcal{L}\pr{\mathbb{H};\mathbb{H}}}^2\leq\underset{k\geq 1}{\sum}\mu_k\norm{e_k}_\infty^2\pr{\norm{e_k}_\infty^2+4C^2_{\mathcal{H}^{1}\subset \mathbb{L}^{\frac{2d}{d-2}}}\norm{\nabla e_k}_d^2}, 
\end{cases}
\end{align}where $\ \mathbb{H}\in\set{\mathcal{H}^{-1},H^{-1},H^{-1}_\nu,\ \nu\in (0,1]}$.
The last inequality follows as before. Let us explain the reasoning in the $H^{-1}$ case.  Provided that $\phi\in C^\infty_0\pr{\mathbb{R}^d}$, whenever $\norm{\phi}_{H^1}\leq 1$, \begin{align*}
\norm{e_k^2\phi}_{H^{1}}^2&=\norm{e_k^2\phi}_{\mathbb{L}^2}^2+\norm{\nabla \pr{e_k^2\phi}}_{\mathbb{L}^2}^2\leq \pr{\norm{\phi}_{2}^2\norm{e_k}_\infty^4+4\norm{e_k}_\infty^2\norm{\nabla e_k}_{\mathbb{L}^d}^2\norm{\phi}_{\mathbb{L}^{\frac{2d}{d-2}}}^2+\norm{e_k}^4_{\infty}\norm{\nabla\phi}_2^2}\\
&\leq \norm{e_k}^2_{\infty}\pr{\norm{e_k}^2_{\infty}+4\norm{\nabla e_k}_{\mathbb{L}^d}^2\norm{\phi}_{\mathbb{L}^{\frac{2d}{d-2}}}^2}\leq \norm{e_k}^2_{\infty}\pr{\norm{e_k}^2_{\infty}+4C_{H^1\subset\mathbb{L}^{\frac{2d}{d-2}}}^{2}\norm{\nabla e_k}_d^2}.
\end{align*}
Then, \begin{align*}\norm{e_k^2x}_{H^{-1}}^2&\leq \norm{x}^2_{H^{-1}}\sup_{\norm{\phi}_{H^1}\leq 1}\norm{e_k^2\phi}_{H^1}\leq \norm{e_k}^2_{\infty}\pr{\norm{e_k}^2_{\infty}+4C_{H^1\subset\mathbb{L}^{\frac{2d}{d-2}}}^{2}\norm{\nabla e_k}_d^2}\norm{x}^2_{H^{-1}},\end{align*}leading to the inequality.
From now on, we set \begin{align}
\label{Const}
C_0:=C_{\mathcal{H}^{-1}\subset H^{-1}}^2\sum_{k\geq 1}\mu_k\pr{\norm{e_k}^2_{\infty}\vee 1}\pr{\norm{e_k}^2_{\infty}+4C^2_{\mathcal{H}^1\subset\mathbb{L}^{\frac{2d}{d-2}}}\norm{\nabla e_k}_d^2}.
\end{align}

\section{The main result}

In this paper we shall prove the existence and the uniqueness of solutions to the equation (\ref{equ}) in the sense of the
definition below.

\begin{definition} Fix any $T>0$.
Let $x\in L^{2}\left( \mathbb{R}^{d}\right) \cap \mathcal{H}^{-1}$. An $\mathcal{H}^{-1}$ - valued adapted process $%
X$ is called strong solution to (\ref{equ}) if the following conditions hold

$X(t,\xi,\omega)>0,$ $dt \times d \xi \times d \mathbb{P}$-a.e. on $\left[ 0,T\right] \times \mathbb{R}^{d}\times \Omega $,

$X\in L^{2}\left( \Omega \times \left( 0,T\right) \times \mathbb{R}
^{d};\mathbb{R} \right) \cap C_{\mathbb{P}}([0,T]; \mathcal{H}^{-1}),$

$ln \left( X\left( \cdot \right) \right) \in L^2\left( \left[ 0,T %
\right] \times \Omega ;\mathcal{H}^{1}\right) $, 

and 

\begin{equation*}
\left\langle X\left( t\right) ,e_{j}\right\rangle
_{2}=\left\langle x,e_{j}\right\rangle _{2}-\int_{0}^{t}\int_{%
\mathbb{R}
^{d}}\left\langle \nabla \ln \left( X\right) ,\nabla
e_{j}\right\rangle _{\mathbb{R}^{d}}d\xi ds+ \left\langle \int_{0}^{t}X\left( s\right)\circ dW _{s},e_{j}\right\rangle _{2}\label{def}
\end{equation*}
for all $j\in\mathbb{N}$ and all $t\in[0,T]$ where $\lbrace e_{j} \rbrace$ is the above orthonormal basis.

\end{definition}

The object of interest in this equation will be the $\ln$ function on its natural domain $\mathbb{R}_+^*:=\{r\in \mathbb{R}; r>0\}$.  For notation purposes, we extend it into a set-valued function $\Psi:\mathbb{R}\rightarrow\mathbb{R}$ by setting \begin{align*}\Psi(r)=\begin{cases}\set{\ln r},&\ \forall r\in\mathcal{D}(\Psi):=\mathbb{R}^*_+,\\
\emptyset,&\ \textnormal{otherwise}.\end{cases}\end{align*}  
It is easy to note that the potential function associated with $\Psi$ is given by the (extended-) real-valued function $j:\mathbb{R}\rightarrow\mathbb{R}\cup\set{+\infty}$, \[j(r):=\begin{cases}r \ln r-r,&\textnormal{ if }r\in\mathbb{R}^*_+,\\
0, &\textnormal{ if }r=0,\\
+\infty,&\textnormal{ otherwise}.\end{cases}\]  One easily notes $j$ to be convex and lower semi-continuous. 

\begin{remark} 
One can easily see that our definition of solution is similar to the usual ones for porous media equation. It is a strong solution from the stochastic point of view and a variational solution from the PDE point of view. 

\end{remark}

We give now the main result of this paper.

\begin{thm}
For each $x\in L^{2}\left( \mathbb{R}^{d}\right) \cap \mathcal{H}^{-1}$ such
that $x\ln x -x \in L^{1}\left( \mathbb{R}^{d}\right) $ and $x>0$ a.e. on $\mathbb{R}^{d}$, there is a unique positive solution $X$ to (\ref{equ}) in the sense of the definition above.
\end{thm}

\section{Proof of the Main Result}

\textbf{Step 1.} As announced in the introduction, we shall take three successive approximations of the solution.\\
We first consider the equation driven by the regularized version of $\Psi$. Namely, having fixed $\lambda>0$, we set 
\begin{equation}
\tilde{\Psi}_\lambda(r):=\Psi_\lambda(r)-\Psi_\lambda(0)+\lambda r,
\label{aprox1}
\end{equation}
for all $r\in\mathbb{R}$.

If  Moreau's theorem applies such that the sub-differential of the potential $j$ (i.e.  $\Psi$) can be approximated via the gradients of the inf-convolutions of $j$ denoted by \[j_\lambda(r):=\inf_{\bar{r}\in\mathbb{R}}\set{j\pr{\bar{r}}+\frac{1}{2\lambda}\abs{r-\bar{r}}^2},\ \forall \lambda>0,\]
we can set $\Psi_\lambda:=\nabla j_\lambda$ and we recall that this corresponds to Yosida's approximations of $\Psi$ i.e.  \[\Psi_\lambda=\frac{1}{\lambda}\pr{\mathbb{I}-\pr{\mathbb{I}+\lambda\Psi}^{-1}}=\Psi \circ\pr{\mathbb{I}+\lambda\Psi}^{-1},\]
which appears in the approximation (\ref{aprox1}).

In the second term of (\ref{aprox1}), $\Psi_\lambda(0)$ was introduced in order to get $\tilde{\Psi}_\lambda(0)=0$ and the final term is necessary in order to have the strong monotonicity property for our operator.

We consider the first approximating equation
\begin{equation}
\label{Eq_lambda}
\left\lbrace 
\begin{split} dX_\lambda(t)&=\bigtriangleup \tilde{\Psi}_\lambda\pr{X_\lambda(t)}\ dt\ +\frac{1}{2}\pr{\sigma\otimes\sigma}\pr{X_\lambda(t)}\ dt+X_\lambda(t)\ dW (t),\ t\geq 0;\\
X_\lambda(0)&=x.\end{split}\right.
\end{equation}

\textbf{Step 2.} Moreover, we consider a further perturbation of the operator $\bigtriangleup$ i.e.  $\bigtriangleup-\nu \mathbb{I}$, for $\nu>0$ and $\mathbb{I}$ being the identity operator. Whenever convenient, we will drop this $\mathbb{I}$ operator and merely write $\bigtriangleup-\nu$. As a consequence, we introduce
\begin{equation}
\label{Eq_lambda_nu}
\left\lbrace 
\begin{split} dX_{\lambda,\nu}(t)&=\pr{\bigtriangleup-\nu} \tilde{\Psi}_\lambda\pr{X_{\lambda,\nu}(t)}\ dt\ +\frac{1}{2}\pr{\sigma\otimes\sigma}\pr{X_{\lambda,\nu}(t)}\ dt+X_{\lambda,\nu}(t)\ dW(t),\ t\geq 0;\\
X_{\lambda,\nu}(0)&=x.\end{split}\right.
\end{equation}
The arguments on the well-posedness of the equation (\ref{Eq_lambda_nu}) follow from \cite[Chapter 4]{Prevot-Rockner}.  
To cope with the notations in \cite[Chapter 4]{Prevot-Rockner}, we conveniently denote by \begin{align*}A&:=\pr{\bigtriangleup-\nu} \tilde{\Psi}_\lambda+\frac{1}{2}\pr{\sigma\otimes\sigma}:V:=\mathbb{L}^{2}\pr{\mathbb{R}^d}\rightarrow V^*, \\ B&:=\sigma:V\rightarrow \mathcal{L}_2\pr{Q^{\frac{1}{2}}\mathcal{H}^{-1}\pr{\mathbb{R}^d};H_\nu^{-1}\pr{\mathbb{R}^d}},\end{align*}two deterministic operators.  
For subsequent developments, we will also employ
\[A_\nu:=\pr{\bigtriangleup-\nu} \tilde{\Psi}_\lambda:V:=\mathbb{L}^{2}\pr{\mathbb{R}^d}\rightarrow V^*,\] 
the operator without the Stratonovich contribution to the drift.  In this sense, the equation can be considered either as an $A$-driven one or as an $A_\nu$-driven, $\frac{1}{2}\sigma\otimes\sigma$-drift perturbed one.\\
The Gelfand triple is based on the inclusion $V:=\mathbb{L}^2\pr{\mathbb{R}^d}\subset H^{-1}_\nu \subset V^*$ (see also Remark \ref{RemIncl}, assertion 1 in the case $\nu=1$).  In particular, one recalls (see \cite[Example 4.1.11]{Prevot-Rockner} and subsequent assertions)\begin{align}\label{V*V}\scal{u,x}_{\pr{V,V^*}}=\scal{u,x}_{H_\nu^{-1}}
=\scal{\pr{\nu-\bigtriangleup}^{-\frac{1}{2}}u, \pr{\nu-\bigtriangleup}^{-\frac{1}{2}}x}_2,\end{align}as soon as $x\in H_\nu^{-1}$(for this particularly useful equality, the reader is referred to \cite[Remark 4.1.14]{Prevot-Rockner} by bearing in mind the operator used in our setting i.e. $\bigtriangleup-\nu$).
For our readers' sake, the main elements of proof are provided in the Appendix \ref{A1}. 
According to \cite[Theorem 4.2.4]{Prevot-Rockner}, the equation \eqref{Eq_lambda_nu} admits a unique solution denoted by $X_{\lambda,\nu}$ 
such that 
\begin{enumerate}
\item $X_{\lambda,\nu}$ is $\mathbb{F}$-adapted and with continuous $H^{-1}_\nu$-valued trajectories; 
\item $X_{\lambda,\nu} \in \mathbb{L}^2\pr{\pp{0,T}\times \Omega;\mathbb{L}^2\pr{\mathbb{R}^d}}$
thus belongs to $\mathbb{L}^2\pr{\pp{0,T}\times\Omega;H^{-1}_\nu}$, too. 
\item $\mathbb{E}\pp{\sup_{t\in\pp{0,T}}\norm{X_{\lambda,\nu}(t)}^2_{H_\nu^{-1}}}<\infty.$
\end{enumerate}

\textbf{Step 3.} The approximating solution $X_{\lambda,\nu}^{\varepsilon}$.\\
At this point, let us introduce the following SDE \begin{equation}\label{Eq_lne}
\left\lbrace \begin{split}dX_{\lambda,\nu}^{\varepsilon}(t)&=-A_\nu^{\varepsilon}\pr{X_{\lambda,\nu}^{\varepsilon}(t)}\ dt\ +\sigma(X_{\lambda,\nu}^{\varepsilon}(t))\ \circ\ dW^Q(t),\ t\geq 0;\\
X_{\lambda,\nu}^{\varepsilon}(0)&=x.
\end{split}\right.
\end{equation}
We recall that the Yosida approximation $A_\nu^\varepsilon:=\frac{1}{\varepsilon}\pr{\mathbb{I}-\pr{\mathbb{I}+\varepsilon A_\nu}^{-1}}$ and the resolvent $\mathbb{J}_{\varepsilon}:=\pr{\mathbb{I}+\varepsilon A_\nu}^{-1}$ are Lipschitz-continuous in this context both in $H^{-1}\pr{\mathbb{R}^d}$ and in $\mathbb{L}^2\pr{\mathbb{R}^d}$.

We shall now pass to the limit in the following order $\varepsilon$, $\nu$ and finally $\lambda$.

\textbf{Step 4.} 
The passage to the limit for $\varepsilon$ is based on the following result.

\begin{proposition}\label{PropStep2}
Let $T>0$ be a finite and fixed time horizon. Then, the following assertions hold true. 
\begin{enumerate}
\item There exists a generic constant $C$ only depending on $T$ and $C_0$, but not of $\varepsilon,\ \lambda,\ \nu$ such that 
\begin{equation}\label{Estim1}\begin{cases} 
\mathbb{E}\pp{\underset{0\leq t\leq T}{\sup}\norm{X^\varepsilon_{\lambda,\nu}(t)}^2_{H^{-1}_\nu}}\leq C\norm{x} ^2_{H^{-1}_\nu};\ \mathbb{E}\pp{\underset{0\leq t\leq T}{\sup}\norm{X^\varepsilon_{\lambda,\nu}(t)}^2_{\mathbb{L}^2}}\leq C\norm{x} ^2_{\mathbb{L}^2};\\
\mathbb{E}\pp{\underset{0\leq t\leq T}{\sup}\norm{X_{\lambda,\nu}(t)}^2_{H^{-1}_\nu}}\leq C\norm{x} ^2_{H^{-1}_\nu};\ 
\underset{0\leq t\leq T}{ess\ sup}\ \mathbb{E}\pp{\norm{X_{\lambda,\nu}(t)}^2_{\mathbb{L}^2}}\leq C\norm{x} ^2_{\mathbb{L}^2}.
\end{cases}
\end{equation}
\item One has the convergence $X_{\lambda,\nu}^{\varepsilon}\underset{\varepsilon\rightarrow 0}{\rightarrow} X_{\lambda,\nu}$ strongly in $\mathbb{L}^\infty\pr{\pp{0,T};\mathbb{L}^2\pr{\Omega; H^{-1}\pr{\mathbb{R}^d}}}$ and weakly in $\mathbb{L}^\infty\pr{\pp{0,T};\mathbb{L}^2\pr{\Omega;\mathbb{L}^2\pr{\mathbb{R}^d}}}$.
\item If the initial datum $x$ is non-negative, the solution $X_{\lambda,\nu}^\varepsilon(t)$ 
remains non-negative $\mathbb{P}$-a.s. and for all $t\in\pp{0,T}$.
\end{enumerate}
\end{proposition}

For our readers' comfort, we provide a few elements of proof. Please note that, as a consequence, the solution $X_{\lambda,\nu}$ remains non-negative if the initial datum $x$ is non-negative.

\begin{proof}[Elements of proof for Proposition \ref{PropStep2}]
\begin{enumerate}
\item We begin with the behaviour in $H^{-1}\pr{\mathbb{R}^d}$.\\
Since, in this framework we deal with a Lipschitz-coefficient-driven equation \eqref{Eq_lne}, the solution is unique in $\mathbb{L}^2_{\mathbb{F}}\pr{\Omega;C\pr{\pp{0,T};H^{-1}\pr{\mathbb{R^d}}}}$ (i.e.  an $\mathbb{F}$-adapted process that can be seen as a $C$-valued square integrable random variable). To this process, one applies, for $\nu>0$, the usual (Hilbert-space) Itô formula to get, for the function $\frac{1}{2}\norm{\cdot}_{H^{-1}_\nu}^2$,
\begin{align*}
&\frac{1}{2}{\norm{X^\varepsilon_{\lambda,\nu}(t)}^2_{H^{-1}_\nu}}\\
=&\frac{1}{2}{\norm{x}^2_{H^{-1}_\nu}}+\int_0^t\pr{\scal{-A^{\varepsilon}_\nu X^\varepsilon_{\lambda,\nu}(s),X^\varepsilon_{\lambda,\nu}(s)}_{H^{-1}_\nu}+\frac{1}{2}\scal{\pr{\sigma\otimes\sigma}\pr{X_{\lambda,\nu}^\varepsilon(s)},X_{\lambda,\nu}^\varepsilon(s)}_{H_\nu^{-1}}}\ ds\\&+\frac{1}{2}\int_0^t\norm{\sigma\pr{X^\varepsilon_{\lambda,\nu}(s)}}^2_{\mathcal{L}_2\pr{Q^{\frac{1}{2}}\mathcal{H}^{-1};H^{-1}_\nu}}ds\\&+\int_0^t\scal{X^\varepsilon_{\lambda,\nu}(s),\sigma\pr{X_{\lambda, \nu}^\varepsilon(s)}dW^{Q}(s)}_{H^{-1}_\nu}.
\end{align*}
Taking into account the monotonicity of the operator and due to Burkholder-Davis-Gundy inequality combined with \eqref{ItoIsom},  \eqref{normSigmaH} and \eqref{PropSigSig}, it follows that, for $0\leq t\leq r\leq T$,
\begin{align*}
\mathbb{E}\pp{\sup_{0\leq t\leq r}\norm{X^\varepsilon_{\lambda,\nu}(t)}^2_{H^{-1}_\nu}}\leq C\pr{\norm{x}_{H^{-1}_\nu}^2+\int_0^r\mathbb{E}\pp{\norm{X^\varepsilon_{\lambda,\nu}(s)}^2_{H^{-1}_\nu}}ds},
\end{align*}
and one concludes using Gronwall's inequality. 
The same kind of argument applies directly to the solution $X_{\lambda,\nu}$. 

\item Let us now turn to the estimates in $\mathbb{L}^2\pr{\mathbb{R}^d}$. \\Since, in this case the coefficients are Lispchitz-continuous, the equation is well-posed  and provides a solution in $\mathbb{L}^2_{\mathbb{F}}\pr{\Omega;C\pr{\pp{0,T};\mathbb{L}^2\pr{\mathbb{R^d}}}}$. 
Then, by applying Itô's formula for $\norm{\cdot}_{\mathbb{L}^2}^2$ and using a standard argument (e.g. \cite[Page 908, Line 5]{criticality} guaranteeing that $\scal{A_\nu^\varepsilon X_{\lambda,\nu}^\varepsilon,X_{\lambda,\nu}^\varepsilon}_{\mathbb{L}^2}\geq0$), one is able to find a generic constant $C>0$, still depending on $T$ and $C_0$ but independent of $\varepsilon,\ \nu,\ \lambda$ such that \begin{equation}
\mathbb{E}\pp{\sup_{0\leq t\leq T}\norm{X^\varepsilon_{\lambda,\nu}(t)}^2_{\mathbb{L}^2}}\leq C\norm{x}^2_{\mathbb{L}^2},\ \forall x\in \mathbb{L}^2\pr{\mathbb{R}^d},\ \forall t\in \pp{0,T}.
\end{equation}.

\item The proof of the convergence is identical to the one in \cite[Claim 3 in Lemma 6.4.5]{BDPR_2016}. 

\item The remaining inequality on $\mathbb{L}^2\pr{\mathbb{R}^d}$ norms involving $X_{\lambda,\nu}$ follows by passing to the limit as $\varepsilon\rightarrow 0+$. 
Let us give a few details. Along some subsequence,  $X_{\lambda,\nu}^\varepsilon\rightarrow X_{\lambda,\nu}$, $dt\times\mathbb{P}\times d\xi$-a.s. on $\pp{0,T}\times \Omega\times \mathbb{R}^d$.  In particular, it follows that $X_{\lambda,\nu}\in\mathbb{L}^2\pr{\Omega\times \pp{0,T};\mathbb{L}^2\pr{\mathbb{R}^d}}$. On the other hand, due to the proofs in item 2., one has 
$\frac{1}{b-a}\int_a^b\mathbb{E}\pp{\norm{X_{\lambda,\nu}^\varepsilon(s)}^2_{\mathbb{L}^2}}dr\leq C\norm{x}_{\mathbb{L}^2}^2,$ (for all $0\leq a<b\leq T$). Passing to the limit along the aforementioned sequence,  $\frac{1}{b-a}\int_a^b\mathbb{E}\pp{\norm{X_{\lambda,\nu}(s)}_{\mathbb{L}^2}^2}dr\leq C\norm{x}_{\mathbb{L}^2}^2,$ for all $0\leq a<b\leq T$. Then, for every Lebesgue point of $\mathbb{E}\pp{\norm{X_{\lambda,\nu}(\cdot)}_{\mathbb{L}^2}^2}$ (hence,  almost surely on $\pp{0,T}$), $\mathbb{E}\pp{\norm{X_{\lambda,\nu}(t)}_{\mathbb{L}^2}^2}\leq C\norm{x}_{\mathbb{L}^2}^2$.

\item Finally, to prove non-negativeness, one works in $\mathbb{L}^2\pr{\mathbb{R}^d}$ and considers the closed subspace $K:=\mathbb{L}^2\pr{\mathbb{R}^d;\mathbb{R}_+}$.  It is clear that the projector $\Pi_K(x)=x^+$ ($\mathbb{P}$-a.s.) is single-valued and the distance $d_K^2(x)=\norm{x^-}_{\mathbb{L}^2}^2$ is regular. 
Here, $x^{-} := -x$ if $x<0$, and $0$ if $x \ge 0$ and $x^{+} := x$ if $x\ge 0$, and $0$ if $x < 0$,  thus $x=x^+ -x^-$.
Then, one can apply the result in \cite[Theorem 3.5]{BQT_2008} to get a sufficient condition (cf.  \cite[Section 4.3]{BQT_2008}) reading
\[-\sum_{k\geq 1}\mu_k\norm{e_kx^-}_{\mathbb{L}^2}^2+2\scal{-A_\nu^\varepsilon x+\frac{1}{2}\sigma\otimes\sigma(x),x^-}_2+C\norm{x^-}_{\mathbb{L}^2}^2\geq 0,\]for some $C>0$ and every $x\in\mathbb{L}^2\pr{\mathbb{R}^d}$. One notes that $\scal{\sigma\otimes\sigma(x),x^-}_2=\underset{k\geq 1}{\sum}\mu_k\scal{e_kx,e_kx^-}_2=-\underset{k\geq 1}{\sum}\mu_k\norm{e_kx^-}_{\mathbb{L}^2}^2$.  In view of the inequality \eqref{normL2}, one only needs to show that $\scal{A_\nu^\varepsilon x,x^-}_2\leq C\norm{x^-}_{\mathbb{L}^2}^2$ (for some generic constant $C>0$). We recall that $\scal{\pr{J_\varepsilon x}^-,x^-}_{2}\leq \norm{x^-}_2^2$ (see \cite[Eq.  (2.81)]{BDPR_2016} such that 
\begin{align*}
\scal{A_\nu^\varepsilon x,x^-}_2&=\scal{\frac{1}{\varepsilon}(x-J_\varepsilon x),x^-}_2=-\frac{1}{\varepsilon}\norm{x^-}_{\mathbb{L}^2}^2-\scal{\frac{1}{\varepsilon}J_\varepsilon x,x^-}_2\\
&\leq -\frac{1}{\varepsilon}\norm{x^-}_{\mathbb{L}^2}^2+\scal{\frac{1}{\varepsilon}\pr{J_\varepsilon x}^-,x^-}_2\leq 0.
\end{align*}
\end{enumerate}
\end{proof}

\textbf{Step 5.} The passage to the limit as $\nu\rightarrow 0+$ is based on the following result. 
In this section we obtain also some estimates which are necessary for the next step.
\begin{proposition}\label{PropStep3}For $T>0$ fixed, the following hold true.
\begin{enumerate}
\item For every $\lambda>0$, $X_{\lambda,\nu}$ converges, as $\nu\rightarrow 0+$ to $X_\lambda$ strongly in $\mathbb{L}^{2}\pr{\Omega;C\pr{\pp{0,T};H^{-1}}}$;
\item There exists a generic constant $C>0$ independent of $\lambda>0$ and of the initial datum $x$ such that \begin{equation}
\label{Estim_lambda_1}
\begin{cases}
(a)\ \mathbb{E}\pp{\underset{0\leq t\leq T}{\sup}\norm{X_{\lambda}(t)}_{\mathcal{H}^{-1}}^2}\leq C\norm{x}_{\mathcal{H}^{-1}}^2;\\
(b)\ \underset{0\leq t\leq T}{ess\ sup}\ \mathbb{E}\pp{\norm{X_{\lambda}(t)}_{\mathbb{L}^2}^2}\leq C\norm{x}_{\mathbb{L}^2}^2.
\end{cases}
\end{equation}
\item Recall that the initial datum $x$ satisfies $x\ln x-x\in \mathbb
{L}^1\pr{\mathbb{R}^d}$ by assumption. Then, the family $\pr{\Psi_\lambda(X_\lambda)}_{\lambda\in(0,1]}$ is (equi-)bounded in $\mathbb{L}^2\pr{\Omega\times \pp{0,T};\mathcal{H}^{1}\pr{\mathbb{R}^d}}$ i.e.  \begin{equation}\label{Estim_lambda_2}
\sup_{\lambda\in(0,1]}\norm{\Psi_\lambda(X_\lambda)}_{\mathbb{L}^2\pr{\Omega\times \pp{0,T};\mathcal{H}^{1}\pr{\mathbb{R}^d}}}^2\leq C\pr{1+\norm{x}_{\mathbb{L}^2}^2+\int_{\mathbb{R}^d}\pr{x\ln x-x}\pr{\xi}d\xi}.
\end{equation}
\end{enumerate}
\end{proposition}

\begin{proof}[Proof of Proposition \ref{PropStep3}]
1. To avoid useless complications, we will simply fix $\lambda>0$ and write $X:=X_{\lambda,\nu}$ and $Y:=X_{\lambda,\nu'}$.
\begin{equation}\label{S3Eq1}
\begin{split}
d\ \pr{X-Y}(t)=&\pr{ \pr{\Delta-\nu}\tilde{\Psi}_\lambda(X(t))-\pr{\Delta-\nu'}\tilde{\Psi}_\lambda(Y(t))}\ dt\\&+\frac{1}{2}\sigma\otimes\sigma\pr{X(t)-Y(t)}\ dt+\pr{\sigma\pr{X(t)-Y(t)}}\ dW^Q(t).
\end{split}
\end{equation}
We write \begin{align*}\pr{ \pr{\Delta-\nu}\tilde{\Psi}_\lambda(x)-\pr{\Delta-\nu'}\tilde{\Psi}_\lambda(y)}=&\pr{\Delta-1}\pr{\tilde{\Psi}_\lambda(x)-\tilde{\Psi}_\lambda(y)}\\&+\pr{1-\nu}\pr{\tilde{\Psi}_\lambda(x)-\tilde{\Psi}_\lambda(y)}-\pr{\nu-\nu'}\tilde{\Psi}_\lambda(y).
\end{align*}
With this in mind, Itô's formula for the process in \eqref{S3Eq1} and the functional $\norm{\cdot}_{H^{-1}}^2$ together with the strict $\lambda$-monotonicity of $\tilde{\Psi}_\lambda$ yields
\begin{align*}
&\norm{(X-Y)(t)}^2_{H^{-1}}+\lambda\int_0^t\norm{X(s)-Y(s)}_{\mathbb{L}^2}^2ds\\[5pt]
&\leq \int_0^t\set{\abs{\scal{\tilde{\Psi}_\lambda(X(s))-\tilde{\Psi}_\lambda(Y(s)),X(s)-Y(s)}_{H^{-1}}}}ds\\[5pt]
&+\int_0^t\set{\abs{\nu-\nu'} \int_0^t\abs{\scal{\tilde{\Psi}_\lambda(Y(s)),X(s)-Y(s)}_{H^{-1}}}}ds\\[5pt]
&+C\int_0^t\norm{X(s)-Y(s)}_{H^{-1}}^2ds+\abs{\int_0^t\scal{X(s)-Y(s),\pr{\sigma(X(s))-\sigma(Y(s))}dW^Q(s)}_{H^{-1}}}
\end{align*}
(Again, the constant $C$ comes from \eqref{ItoIsom} and \eqref{PropSigSig}.) 

Let us explain how to deal with the term $\abs{\scal{\tilde{\Psi}_\lambda(X(s))-\tilde{\Psi}_\lambda(Y(s)),X(s)-Y(s)}_{H^{-1}}}$. 

Using the Lipschitz continuity of $\tilde{\Psi}_\lambda$, one finds a constant $C>0$ (possibly depending on $\lambda>0$ but independent of $t\leq T$ and of $\nu,\nu'$) such that \begin{align*}
\abs{\scal{\tilde{\Psi}_\lambda(x)-\tilde{\Psi}_\lambda(y),x-y}_{H^{-1}}} &\leq \norm{\tilde{\Psi}_\lambda(x)-\tilde{\Psi}_\lambda(y)}_{\mathbb{L}^2}\norm{x-y}_{H^{-1}}  \\[5pt] 
&\leq C\norm{x-y}_{\mathbb{L}^2}\norm{x-y}_{H^{-1}}\\[5pt] 
&\leq \frac{\lambda}{2}\norm{x-y}_{\mathbb{L}^2}^2+C\norm{x-y}_{H^{-1}}^2.
\end{align*}
Using the linear bound $\tilde{\Psi}_\lambda$ and Burkholder-Davis-Gundy inequality to deal with the $W^Q$-term, we get the existence of a constant $C>0$ (possibly depending on $\lambda>0$ but independent of $t\leq T$ and of $\nu,\nu'$) such that, for $0\leq t\leq r\leq T$, 
\begin{equation*}
\begin{split}
&\mathbb{E}\pp{\sup_{0\leq t\leq r}\norm{\pr{X-Y}(t)}_{H^{-1}}^2}+\mathbb{E}\pp{\int_0^r\norm{X(s)-Y(s)}_{\mathbb{L}^2}^2ds}\\[5pt]
&\leq C\mathbb{E}\pp{\int_0^t\norm{X(s)-Y(s)}_{H^{-1}}^2ds}+C\abs{\nu-\nu'}\mathbb{E}\pp{\int_0^r\norm{Y(t)}_{\mathbb{L}^2}^2ds}.
\end{split}
\end{equation*}
Then, thanks to Gronwall's inequality and to the estimates in Proposition \ref{PropStep2} 1., we get 
\begin{equation}\label{S3Eq2}
\mathbb{E}\pp{\sup_{0\leq t\leq T}\norm{\pr{X_{\lambda,\nu}(t)-X_{\lambda,\nu'}(t)}}_{H^{-1}}^2}+\mathbb{E}\pp{\int_0^T \norm{X_{\lambda,\nu}(s)-X_{\lambda,\nu'}(s)}_{\mathbb{L}^2}}\leq C\abs{\nu-\nu'}\norm{x}_{\mathbb{L}^2}.
\end{equation}
The first assertion follows from the completeness of the underlying space.\\

2. To see that the limit $X_\lambda$ belongs to $\mathcal{H}^{-1}$, one proceeds as folows.  We recall that, due to \eqref{Estim1},
\begin{align*}
\mathbb{E}\pp{\sup_{0\leq t\leq T}\norm{X_{\lambda,\nu}(t)}_{H_{\nu}^{-1}}^2}\leq C\norm{x}^2_{H_{\nu}^{-1}}.
\end{align*}
On the other hand, since $\mu\mapsto \norm{\cdot}_{H^{-1}_\mu}$ is non-increasing,  \begin{align*}
\mathbb{E}\pp{\sup_{0\leq t\leq T}\norm{X_{\lambda,\nu}(t)}_{H_{\mu}^{-1}}^2}\leq C\norm{x}^2_{H_{\nu}^{-1}},\ \forall \nu\leq \mu.
\end{align*}
Having fixed $\mu>0$, due to the first assertion,  
\begin{equation}
X_{\lambda,\nu}\rightarrow X_\lambda \textit{ as } \nu\rightarrow 0+ \quad \mbox{ in} \quad  \mathbb{L}^{2}\pr{\Omega;C\pr{\pp{0,T};H^{-1}_\mu}}
\end{equation}
(remember $\norm{\cdot}_{H^{-1}}$ and $\norm{\cdot}_{H^{-1}_\mu}$ are equivalent).  Then, passing to the limit as $\nu\rightarrow0+$, one gets 
\begin{align*}
\mathbb{E}\pp{\sup_{0\leq t\leq T}\norm{X_{\lambda}(t)}_{H_{\mu}^{-1}}^2}\leq C\norm{x}^2_{\mathcal{H}^{-1}}. 
\end{align*}
Finally, let us note that
\[\underset{\mu\rightarrow0+}{\lim}\sup_{0\leq t\leq T}\norm{X_{\lambda}(t)}_{H^{-1}_\mu}=\sup_{\mu>0+}\sup_{0\leq t\leq T}\norm{X_{\lambda}(t)}_{H^{-1}_\mu}=\sup_{0\leq t\leq T}\sup_{\mu>0+}\norm{X_{\lambda}(t)}_{H^{-1}_\mu}=\sup_{0\leq t\leq T}\norm{X_{\lambda}(t)}_{\mathcal{H}^{-1}},\]and, by the monotone convergence theorem, one gets the first inequality in \eqref{Estim_lambda_1}. \\
For the second one, one uses almost sure convergence of (some subsequence of) $X_{\lambda,\nu}$ to $X_\lambda$ combined with Fatou's lemma and the upper estimates in Proposition \ref{PropStep2}.\\
3. Since we are interested in terms like $\abs{\nabla\Psi_\lambda\pr{X_\lambda}}^2$, we introduce
\begin{equation*}
\Phi _{\lambda }(x):=\int_{\mathbb{R}^{d}}\pr{j_{\lambda }\left(
x\left( \xi \right) \right)+\frac{\lambda}{2}\pr{x(\xi)}^2} d\xi ,\ \forall x\in \mathbb{L}^{2}\left( \mathbb{R} 
^{d}\right),
\end{equation*}such that the $\mathbb{L}^2$ Fréchet-derivative of  $\Phi _{\lambda }(x)$ provides 
$\Psi_\lambda(x)+\lambda x$.

We apply It\^{o}'s formula with $\Phi _{\lambda }$ (the computation may be justified as in Step 3) to the $\mathbb{L}^2\pr{\mathbb{R}^d}$-valued process $X_\lambda$ to get 
\begin{equation*}
\begin{split}
&\mathbb{E}\pp{\Phi_\lambda\pr{X_\lambda(t)}}\\[5pt]
&=\mathbb{E}\pp{\Phi_\lambda\pr{x}}+\mathbb{E}\pp{\int_0^t\scal{\Delta\tilde{\Psi}_\lambda\pr{X_\lambda(s)},\tilde{\Psi}_\lambda\pr{X_\lambda(s)}+\Psi_\lambda\pr{0}}_{\mathbb{L}^2}ds}\\[5pt] 
&+\frac{1}{2}\sum_{k\geq 1}\mu_k\mathbb{E}\pp{\int_0^t\int_{\mathbb{R}^d}X_\lambda(s)\pr{\tilde{\Psi}_\lambda\pr{X_\lambda(s)}+\Psi_\lambda(0)}e_k^2d\xi ds}\\[5pt]
&+\frac{1}{2}\sum_{k\geq 1}\mu_k\mathbb{E}\pp{\int_0^t\int_{\mathbb{R}^d}\tilde\Psi_\lambda'\pr{X_\lambda(s)}X_\lambda^2(s)e_kd\xi ds}.
\end{split}
\end{equation*}
Let us now focus on the term $\tilde\Psi_\lambda'\pr{X_\lambda(s)}$.  
Using chain rules, one easily computes 
\begin{equation}
\tilde\Psi_\lambda'\pr{r}=\lambda+\frac{1}{\lambda+J_\lambda(r)}
\end{equation}

where $J_\lambda(r)=\pr{\mathbb{I}+\lambda\Psi}^{-1}(r)$.  

Using \eqref{normL2}, it follows that 
\begin{equation}\label{S3Eq3}
\begin{split}
&\mathbb{E}\pp{\Phi_\lambda\pr{X_\lambda(t)}}+\mathbb{E}\pp{\int_0^t\norm{\nabla\tilde{\Psi}_\lambda\pr{X_\lambda(s)}}_{\mathbb{L}^2}^2ds}\\[5pt]
&= \mathbb{E}\pp{\Phi_\lambda\pr{x}}+\frac{1}{2}\sum_{k\geq 1}\mu_k \big\lbrace\mathbb{E}\pp{\int_0^t\int_{\mathbb{R}^d}\pr{\Psi_{\lambda}\pr{X_\lambda(s)}+\lambda X_{\lambda}(s)}X_\lambda(s)e_k^2d\xi ds}\\[5pt]
&+\mathbb{E}\pp{\int_0^t\int_{\mathbb{R}^d}\frac{X_\lambda^2(s)}{\lambda+J_\lambda\pr{X_\lambda(s)}}e_k^2d\xi\ ds}+\mathbb{E}\pp{\int_0^t\lambda\norm{X_\lambda(s)e_k^2}_{\mathbb{L}^2}^2ds} \big\rbrace.
\end{split}
\end{equation}

We can prove that 
\begin{proposition}\label{PropAux}The application $\mathbb{R}_+\ni r\mapsto \Psi_\lambda(r)+\frac{r}{\lambda+J_\lambda(r)}-2r$ is non-positive (for $\lambda\leq\frac{1}{2}$).
\end{proposition}

The proof of Proposition \ref{PropAux} relies on some technical computation and will be relegated to the Appendix.

Going back to \eqref{S3Eq3}, one has for $\lambda \le 1/2$
\begin{equation}
\label{S3Eq4}
\begin{split}
&\mathbb{E}\pp{\Phi_\lambda\pr{X_\lambda(t)}}+\mathbb{E}\pp{\int_0^t\norm{\nabla\tilde{\Psi}_\lambda\pr{X_\lambda(s)}}_{\mathbb{L}^2}^2ds}\\[5pt]
&\leq \mathbb{E}\pp{\Phi_\lambda\pr{x}}+\sum_{k\geq 1}\mu_k\norm{e_k}_{\infty}^2(\lambda+1)\mathbb{E}\pp{\int_0^t\norm{X_\lambda(s)}_{\mathbb{L}^2}^2ds}.
\end{split}
\end{equation}
Our assertion follows thanks to \eqref{Estim_lambda_1} (b) by noting that $\Phi_\lambda$ is lower bounded and $j_\lambda(r)\leq j(r)\leq r\ln r-r$, for $\lambda>0$.
\end{proof}

\textbf{Step 6.} The passage to the limit as $\lambda\rightarrow 0+$ 

As a consequence of the estimates form the previous step, 
the following weak convergences hold true.
\begin{align}\label{WeakConv_1}
\begin{cases}
X_\lambda{\underset{\lambda\longrightarrow 0}{\rightarrow}}X,\ \textnormal{weakly-* in } \mathbb{L}^\infty\pr{\pp{0,T};\mathbb{L}^2\pr{\Omega;\mathcal{H}^{-1}}};\\
X_\lambda{\underset{\lambda\longrightarrow 0}{\rightarrow}}X,\ \textnormal{weakly in }\mathbb{L}^2\pr{\pp{0,T}\times\Omega;\mathbb{L}^{2}};\\
\tilde\Psi_\lambda\pr{X_\lambda}{\underset{\lambda\longrightarrow 0}{\rightarrow}}\eta,\ \textnormal{weakly in }\mathbb{L}^2\pr{\pp{0,T}\times\Omega;\mathcal{H}^{1}}.
\end{cases}
\end{align}

Using the linearity and boundedness in $\mathcal{H}^{-1}$ of $\sigma\otimes\sigma$, one has a (weak) sense of the limiting equation
\begin{align}\label{EqLimit}
\left\langle X\left( t\right) ,e_{j}\right\rangle
_{2}=\left\langle x,e_{j}\right\rangle _{2}-\int_{0}^{t}\int_{%
\mathbb{R}^{d}}\left\langle \nabla \eta \left( X\right) ,\nabla
e_{j}\right\rangle _{\mathbb{R}^{d}}d\xi ds+\left\langle \int_{0}^{t}X\left( s\right)\circ dW _{s},e_{j}\right\rangle _{2}.\end{align}

\textbf{Step 7.} To conclude, one still has to prove that $\eta=\Psi(X)$ in a $\mathbb{P}\times\mathcal{L}eb$ almost sure sense on $\Omega\times \pp{0,T}\times \mathbb{R}^d$. 

This follows from the following result.
\begin{proposition}\label{PropStep5}
The following assertion holds true.
\begin{align}
\label{EqStar}
\underset{\lambda\rightarrow0+}{\lim\sup}\ \mathbb{E}\pp{\int_0^T\int_{\mathbb{R}^d}\tilde{\Psi}_\lambda\pr{X_\lambda(t)(\xi)}X_\lambda(t)(\xi) d\xi dt}\leq \mathbb{E}\pp{\int_0^T\int_{\mathbb{R}^d}\eta(t)(\xi)X(t)(\xi) d\xi dt}.
\end{align}
\end{proposition}
\begin{proof}
First, let us note that if $K$ is a bounded subset of $\mathbb{R}^d$,  and $x\in \mathbb{L}^2\pr{\mathbb{R}^d}$,  non-negative-valued, then $x\mathbf{1}_K\in \mathcal{H}^{-1}$. Indeed, one begins with writing $x(\xi)\mathbf{1}_K(\xi)=x(\xi)\mathbf{1}_{x(\xi)\leq 1,\xi\in K}+x(\xi)\mathbf{1}_{x(\xi)>1,\xi\in K}$.  Then, with $p=\frac{2d}{d+2}$, 
\begin{align*}\int_{\mathbb{R}^d}\pr{x(\xi)\mathbf{1}_K(\xi)}^pd\xi \leq \int_{\mathbb{R}^d}\pr{\mathbf{1}_{\xi\in K}+x^2(\xi)\mathbf{1}_{x(\xi)>1}}d\xi\leq \mathcal{L}eb(K)+\norm{x}_2^2<\infty.\end{align*}As such, $x\mathbf{1}_K\in \mathbb{L}^{\frac{2d}{d+2}}\subset\mathcal{H}^{-1}$. As a consequence, if $Y\in\mathbb{L}^2\pr{\Omega\times\pp{0,T};\mathbb{L}^2\pr{\mathbb{R}^d}}$, then $Y\mathbf{1}_K\in\mathbb{L}^2\pr{\Omega\times\pp{0,T};\mathcal{H}^{-1}}$.\\
Next, for $k\geq 1$, let $K$ be a Lebesgue measure-continuous bounded \footnote{Note that by the density argument, we may assume $e_k \in \mathcal{S}_0 (\mathbb{R}^d) \cap \mathcal{D}(\mathbb{R}^d)$ with the notations from \cite{BCD_2011}.} set (i.e. $\mathcal{L}eb\pr{\partial K}=0$) 
containing the support of $e_k$. 

 Then, if $x\in\mathbb{L}^2\pr{\mathbb{R}^d}\cap\mathcal{H}^{-1}$, then $x\mathbf{1}_K$ is $\mathcal{H}^{-1}$-valued and \begin{equation}\label{Estim+}
\scal{x \mathbf{1}_K,e_k}_{\mathcal{H}^{-1}}=-\scal{x \mathbf{1}_K,\pr{-\Delta}^{-1}e_k}_{2}=-\scal{x,\pr{-\Delta}^{-1}e_k}_{2}=\scal{x,e_k}_{\mathcal{H}^{-1}}
\end{equation}
We claim that $X_\lambda\mathbf{1}_K$ converges strongly in $\mathbb{L}^2\pr{\Omega\times \pp{0,T};\mathcal{H}^{-1}}$ to $X\mathbf{1}_K$.

Let us further note that \begin{enumerate}
\item If $x\in \mathbb{L}^2\pr{\mathbb{R}^d}\cap H^{-1}$ and $y \in \mathbb{L}^2(\mathbb{R}^d)$, then 
\begin{equation*}
\scal{x\mathbf{1}_K,y}_{H^{-1}_\mu}=\scal{\pr{\mu\mathbb{I}-\Delta}^{-1}\pr{\mathbf{1}_Kx},y}_2=\int_K\pr{\mu\mathbb{I}-\Delta}^{-1}x(\xi)y(\xi)d\xi.
\end{equation*}

\item If $x^n$ converges to $x$ strongly in $\mathbb{L}^2\pr{\mathbb{R}^d}$, then, by the same argument seen before, $(x^n-x)\mathbf{1}_K$ converges to $0$ in $\mathbb{L}^{\frac{2d}{d+2}}$ (hence in the larger spaces $\mathcal{H}^{-1}\subset {H}^{-1}$.)
\item The same applies to the convergence of  $X^n$ to $X$ strongly in $\mathbb{L}^2\pr{\Omega\times\pp{0,T}\times\mathbb{R}^d;\mathbb{R}}$ leading to $(X_n-X)\mathbf{1}_K$ converging to $0$ strongly in $\mathbb{L}^2\pr{\Omega\times \pp{0,T}; \mathbb{L}^{\frac{2d}{d+2}}}$;
\end{enumerate}
One applies Itô's formula in $H_{\nu}^{-1}$ to $\pr{X_{\lambda,\nu}^\varepsilon-X_{\lambda',\nu}^\varepsilon}\mathbf{1}_K$. 

In this framework, the term $-A_\nu^\varepsilon\pr{X_{\lambda,\nu}^\varepsilon}+A_\nu^\varepsilon\pr{X_{\lambda',\nu}^\varepsilon}$ is consistent in $\mathbb{L}^2\pr{\mathbb{R}^d}$ such that $\pr{A_\nu^\varepsilon\pr{X_{\lambda,\nu}^\varepsilon}+A_\nu^\varepsilon\pr{X_{\lambda',\nu}^\varepsilon}}\mathbf{1}_K$ belongs to $\mathbb{L}^{\frac{2d}{d+2}}\subset\mathcal{H}^{-1}\subset H^{-1}$. One gets 
\begin{align*}
&\mathbb{E}\pp{\norm{\pr{X_{\lambda,\nu}^\varepsilon(t)-X_{\lambda',\nu}^\varepsilon(t)}\mathbf{1}_K}^2_{H^{-1}_\nu}}\\[5pt]
&+\mathbb{E}\pp{\int_0^t\int_K\pp{\tilde{\Psi}_\lambda\pr{\mathbb{J}_\varepsilon\pr{X_{\lambda,\nu}^\varepsilon(s)}}-\tilde{\Psi}_{\lambda'}\pr{\mathbb{J}_\varepsilon\pr{X_{\lambda',\nu}^\varepsilon(s)}}}\pr{X_{\lambda,\nu}^\varepsilon(s)-X_{\lambda',\nu}^\varepsilon}d\xi\ ds}\\[5pt]
&\leq C\mathbb{E}\pp{\int_0^t\norm{\pr{X_{\lambda,\nu}^\varepsilon(s)-X_{\lambda',\nu}^\varepsilon(s)}\mathbf{1}_K}^2_{H^{-1}_\nu}ds}
\end{align*}
As a consequence, by integrating this on a time interval $\pp{a,b}\subset \pp{0,T}$ and passing to the limit as $\varepsilon\rightarrow 0+$ then as $\nu\rightarrow 0+$ (each of the envolved processes belonging to $\mathbb{L}^2$), one gets the inequality below at every Lebesgue point (hence $\mathbb{P}$-a.s. )

\begin{eqnarray*}
&&\mathbb{E}\frac{1}{2}\left\Vert X_{\lambda }(t)\mathbf{1}_{K}-X_{\lambda
^{\prime }}(t)\mathbf{1}_{K}\right\Vert _{\mathcal{H}^{-1}}^{2}+\mathbb{E}%
\int_{0}^{t}\int_{K}\left( \Psi _{\lambda }\left( X_{\lambda }(s)\right) -\Psi
_{\lambda ^{\prime }}\left( X_{\lambda ^{\prime }}(s)\right) \right) \left(
X_{\lambda }(s)-X_{\lambda ^{\prime }}(s)\right) d\xi ds \\
&&+\mathbb{E}\int_{0}^{t}\int_{K}\left( \lambda X_{\lambda }(s)-\lambda
^{\prime }X_{\lambda ^{\prime }}(s)\right) \left( X_{\lambda }(s)-X_{\lambda
^{\prime }}(s)\right) d\xi ds \\
&\leq &C\mathbb{E}\int_{0}^{t}\left\Vert X_{\lambda }(s)\mathbf{1}%
_{K}-X_{\lambda ^{\prime }}(s)\mathbf{1}_{K}\right\Vert _{\mathcal{H}%
^{-1}}^{2}ds.
\end{eqnarray*}

Since we can write%
\begin{eqnarray*}
&&\left( \Psi _{\lambda }\left( X_{\lambda }\right) -\Psi _{\lambda ^{\prime
}}\left( X_{\lambda ^{\prime }}\right) \right) \left( X_{\lambda
}-X_{\lambda ^{\prime }}\right)  \\[5pt]
&=&\left( \Psi \left( J_{\lambda }\left( X_{\lambda }\right) \right) -\Psi
\left( J_{\lambda ^{\prime }}\left( X_{\lambda ^{\prime }}\right) \right)
\right)  \left( J_{\lambda }\left( X_{\lambda }\right) -J_{\lambda ^{\prime
}}\left( X_{\lambda ^{\prime }}\right) +\lambda \Psi _{\lambda }\left(
X_{\lambda }\right) -\lambda ^{\prime }\Psi _{\lambda ^{\prime }}\left(
X_{\lambda ^{\prime }}\right) \right)  \\[5pt]
&\geq & - \frac{\left( \lambda +\lambda ^{\prime }\right)}{2} \left( \left\vert \Psi
_{\lambda }\left( X_{\lambda }\right) \right\vert ^{2}+\left\vert \Psi
_{\lambda ^{\prime }}\left( X_{\lambda ^{\prime }}\right) \right\vert
^{2}\right) ,
\end{eqnarray*}%
we get by Gronwall's lemma that%
\begin{eqnarray*}
&&\mathbb{E}\left\Vert X_{\lambda }(t)\mathbf{1}_{K}-X_{\lambda ^{\prime }}(t)%
\mathbf{1}_{K}\right\Vert _{\mathcal{H}^{-1}}^{2} \\
&\leq &\left( \lambda +\lambda ^{\prime }\right) \mathbb{E}%
\int_{0}^{t}\int_{K}\left( \left\vert \Psi _{\lambda }\left( X_{\lambda
}(s)\right) \right\vert ^{2}+\left\vert \Psi _{\lambda ^{\prime }}\left(
X_{\lambda ^{\prime }}(s)\right) \right\vert ^{2}\right) d\xi ds \\
&&+\left( \lambda +\lambda ^{\prime }\right) \mathbb{E}\int_{0}^{t}\int_{K}%
\left( \left\vert X_{\lambda }(s)\right\vert ^{2}+\left\vert X_{\lambda
^{\prime }}(s)\right\vert ^{2}\right) d\xi ds.
\end{eqnarray*}

Keeping in mind that 
\[
\mathbb{E}\int_{0}^{t}\int_{K}\left( \left\vert \Psi _{\lambda }\left(
X_{\lambda }(s)\right) \right\vert ^{2}+\left\vert \Psi _{\lambda ^{\prime
}}\left( X_{\lambda ^{\prime }}(s)\right) \right\vert ^{2}\right) d\xi ds<C
\]%
follows from the Poincar\'{e} inequality and (\ref{Estim_lambda_2}),  we get thus that by (\ref{Estim_lambda_1}) 
\[
\mathbb{E}\left\Vert X_{\lambda }(s)\mathbf{1}_{K}-X_{\lambda ^{\prime }}(s)%
\mathbf{1}_{K}\right\Vert _{\mathcal{H}^{-1}}^{2}\underset{\lambda ,\lambda
^{\prime }\rightarrow 0}{\longrightarrow }0.
\]

 In particular,  combined with \eqref{Estim+}, this implies that \begin{equation}\begin{cases}
\scal{X_\lambda,e_k}_{\mathcal{H}^{-1}} \textnormal{ converges strongly in }\mathbb{L}^2\pr{\Omega\times \pp{0,T};\mathbb{R}} \textnormal{ to }\scal{X,e_k}_{\mathcal{H}^{-1}};\\[5pt]
X_\lambda e_k \textnormal{ converges strongly in }\mathbb{L}^2\pr{\Omega\times \pp{0,T};\mathcal{H}^{-1}} \textnormal{ to }X e_k.
\\[5pt]
X_\lambda e_k^2 \textnormal{ converges strongly in }\mathbb{L}^2\pr{\Omega\times \pp{0,T};\mathcal{H}^{-1}} \textnormal{ to }X e_k^2.
\end{cases}\label{Estim++}
\end{equation}

In order to prove (\ref{EqStar}) we shall first apply the It\^{o} formula
with $u\longmapsto \frac{1}{2}\left\Vert u\right\Vert _{\mathcal{H}^{-1}}^{2}
$ to the process $X_{\lambda }$ and we get 
\begin{eqnarray*}
&&\frac{1}{2}\mathbb{E}\left\Vert X_{\lambda }(t)\right\Vert _{\mathcal{H}%
^{-1}}^{2}+\mathbb{E}\int_{0}^{t}\int_{\mathbb{R}^{d}}\Psi _{\lambda
}\left( X_{\lambda }(s)\right) X_{\lambda }(s)d\xi ds+\lambda \mathbb{E}%
\int_{0}^{t}\int_{\mathbb{R}^{d}}\left\vert X_{\lambda }(s)\right\vert
^{2}d\xi ds \\[5pt]
&\leq &\frac{1}{2}\mathbb{E}\left\Vert x\right\Vert _{\mathcal{H}^{-1}}^{2}+C%
\mathbb{E}\int_{0}^{t}\sum_{k=1}^{\infty }\mu _{k}\left\langle X_{\lambda
}\left( s\right) ,e_{k}^{2}X_{\lambda}\left( s\right) \right\rangle _{\mathcal{H}%
^{-1}}ds \\[5pt]
&&+C\mathbb{E}\int_{0}^{t}\sum_{k=1}^{\infty }\mu _{k}\left\vert
X_{\lambda}\left( s\right) e_{k}\right\vert _{\mathcal{H}^{-1}}^{2}ds.
\end{eqnarray*}

By using the Fatou's lemma in the first term and (\ref{Estim++}) in the last part
we get%
\begin{eqnarray*}
&&\frac{1}{2}\mathbb{E}\left\Vert X(t)\right\Vert _{\mathcal{H}^{-1}}^{2}+%
\underset{\lambda \rightarrow 0}{\lim \inf }  \mathbb{E}\int_{0}^{t}\int_{%
\mathbb{R}^{d}}\Psi _{\lambda }\left( X_{\lambda }(s)\right) X_{\lambda }(s)d\xi ds
\\[5pt]
&\leq &\frac{1}{2}\mathbb{E}\left\Vert x\right\Vert _{\mathcal{H}^{-1}}^{2}+C%
\mathbb{E}\int_{0}^{t}\sum_{k=1}^{\infty }\mu _{k}\left\langle X\left(
s\right) ,e_{k}^{2}X\left( s\right) \right\rangle _{\mathcal{H}^{-1}}ds \\[5pt]
&&+C\mathbb{E}\int_{0}^{t}\sum_{k=1}^{\infty }\mu _{k}\left\vert X\left(
s\right) e_{k}\right\vert _{\mathcal{H}^{-1}}^{2}ds.
\end{eqnarray*}

On the other hand, by appling the It\^{o} formula to the solution verifying (\ref{EqLimit}) with
the same norm and combining with the previous relation we get (\ref{EqStar}%
).

\end{proof}

We continue the proof of our main result. For simplicity, let us denote by $\bar\Psi_\lambda(x):=\Psi_\lambda(x)+\lambda x$. 
We consider $\phi_N$ a non-decreasing sequence of infinitely differentiable, $\pp{0,1}$-valued functions such that $\mathbf{1}_{B_N}\leq \phi_N\leq \mathbf{1}_{B_{N+1}}$ for every $N\geq 1$.  

The convexity of $r\mapsto \tilde{j}_\lambda(r):=j_\lambda(r)+\frac{\lambda}{2}r^2$ yields 
\begin{equation*}
\pr{\Psi_\lambda(x)+\lambda x }(x-u)\geq \tilde{j}_\lambda(x)-\tilde{j}_\lambda(u)
\end{equation*}
for every $\pr{x,u}$. 

By applying this inequality to the couple $(X_\lambda, U)$ 
(for a process $U\in \mathbb{L}^2\pr{\Omega\times \pp{0,T};\mathbb{L}^{2,loc}\pr{\mathbb{R}^d}}$) and 
by multiplying with the non-negative $\phi_N$ (for $N$ fixed, for the time being), and by integrating, it follows that\begin{equation}\label{Estim_1}\begin{split}
&\mathbb{E}\pp{\int_0^T\int_{\mathbb{R}^d}\bar{\Psi}_\lambda\pr{X_\lambda(t,\xi)}\pr{X_\lambda(t,\xi)-U(t,\xi)}\phi_N(\xi)d\xi dt}\\[5pt]
&\geq \mathbb{E}\pp{\int_0^T\int_{\mathbb{R}^d}\pr{\tilde j_\lambda\pr{X_\lambda(t,\xi)}-\tilde j_\lambda(U(t,\xi)}\phi_N(\xi)d\xi dt}.
\end{split}\end{equation}
The local integrability of $U$ guarantees the consistency of the left-hand term.  The function $j_\lambda(r)$ is non-positive as $r\leq 1$ and bounded from below by $-1$ if $ 0 \le r \le 1$. It is upper-bounded by $r\ln r$, being sub-quadratic when $r\geq 1$. The bounded support of $\phi_N$ then guarantees consistency of the right-hand term.\\
Again by $0 \le  \phi_N \le 1$ and $\tilde{\Psi}_\lambda(x)x\geq 0$, it follows, from \eqref{Estim_1} that
\begin{equation}\label{Estim_2}\begin{split}
\underset{\lambda\rightarrow 0+}{\lim\inf}\ 
\mathbb{E}&\pp{\int_0^T\int_{\mathbb{R}^d}\tilde{\Psi}_\lambda\pr{X_\lambda(t,\xi)}X_\lambda(t,\xi)d\xi dt}\\[5pt]
&\geq\underset{\lambda\rightarrow 0+}{\lim\inf}\ 
\mathbb{E}\pp{\int_0^T\int_{\mathbb{R}^d}\bar{\Psi}_\lambda\pr{X_\lambda(t,\xi)}X_\lambda(t,\xi)\phi_N(\xi)d\xi dt}\\[5pt]
&\geq \underset{\lambda\rightarrow 0+}{\lim\inf}\ \Big\lbrace \mathbb{E}\pp{\int_0^T\int_{\mathbb{R}^d}\bar{\Psi}_\lambda\pr{X_\lambda(t,\xi)}U(t,\xi)\phi_N(\xi)d\xi dt}\\[5pt]
&+ \mathbb{E}\pp{\int_0^T\int_{\mathbb{R}^d}\tilde j_\lambda\pr{X_\lambda(t,\xi)}\phi_N(\xi)d\xi dt}-\mathbb{E}\pp{\int_0^T\int_{\mathbb{R}^d}\tilde j_\lambda(U(t,\xi))\phi_N(\xi)d\xi dt}\Big\rbrace.
\end{split}\end{equation}
We designate by $I^k_\lambda$ indexed by $k\in\set{1,2,3}$ the three integral terms appearing on the right side. 
\begin{enumerate}
\item For the first term $I^1_\lambda$ one proceeds as follows. 

We recall the boundedness in $\mathbb{L}^2\pr{\Omega\times\pp{0,T};\mathbb{L}^2\pr{\mathbb{R}^d}}$ of $\nabla \tilde\Psi_\lambda\pr{X_\lambda}=\nabla \bar\Psi_\lambda\pr{X_\lambda}$ and the compact support of $\phi_N$ to deduce, by invoking Poincaré's inequality on the bounded open set $B_{N+1}$, that \[\lim_{\lambda\rightarrow 0+}I_\lambda^1=\mathbb{E}\pp{\int_0^T\int_{\mathbb{R}^d}\eta(t,\xi)U(t,\xi)\phi_N(\xi)d\xi dt}.\]
\item For $I_\lambda^3$,  one uses the point-wise convergence of $\tilde j_\lambda$ to $j$ as $\lambda\rightarrow0+$, combined with the aforementioned bounds on $j_\lambda$ and the bounded support of $\phi_N$ to deduce, via Lebesgue's dominated convergence on $\Omega\times \pp{0,T}\times B_{N+1}$,
 \[\lim_{\lambda\rightarrow 0+}I_\lambda^3=\mathbb{E}\pp{\int_0^T\int_{\mathbb{R}^d}j(U(t,\xi))\phi_N(\xi)d\xi dt}.\]

\item For the remaining term, we note that there exists $\lambda_0 <1$ such that
$j_\lambda(x)\geq j_{\lambda_0}(x)$ for every $\lambda\leq \lambda_0$ (and all $x\in\mathbb{R}_+$). The quadratic contribution $\lambda X_\lambda^2$ has a null limit in $\mathbb{P}\times dt\times d\xi$-mean owing to the $\lambda$- uniform bounds on square moments. \\
 Second, the functional $x \in \mathbb{L}^2\mapsto \mathbb{E}\pp{\int_0^T\int_{\mathbb{R}^d}j_{\lambda_0}(x(t,\xi))\phi_N(\xi)d\xi dt}$ is convex and strongly (hence weakly) lower semicontinuous. This is a consequence of $j_{\lambda_0}$ being continuous as a real function and Fatou's Lemma.  As a consequence,
 \begin{align*}\underset{\lambda\rightarrow0+}{\lim\inf}\ \mathbb{E}\pp{\int_0^T\int_{\mathbb{R}^d}j_{\lambda}(X_\lambda(t,\xi))\phi_N(\xi)d\xi dt}&\geq \underset{\lambda\rightarrow0+}{\lim\inf}\ \mathbb{E}\pp{\int_0^T\int_{\mathbb{R}^d}j_{\lambda_0}(X_\lambda(t,\xi))\phi_N(\xi)d\xi dt}\\&\geq \mathbb{E}\pp{\int_0^T\int_{\mathbb{R}^d}j_{\lambda_0}(X(t,\xi))\phi_N(\xi)d\xi dt}.
 \end{align*}

To conclude, one takes the supremum over $\lambda_0>0$ and uses dominated (or monotone) convergence to conclude that \[\underset{\lambda\rightarrow0+}{\lim\inf}\ I_\lambda^2\geq \mathbb{E}\pp{\int_0^T\int_{\mathbb{R}^d}j(X(t,\xi))\phi_N(\xi)d\xi dt}.\]
\end{enumerate}
Plugging these three items into \eqref{Estim_2} and recalling that \eqref{Estim_1} holds true, we finally get 
\begin{equation}\label{Estim_3}\begin{split} &\mathbb{E}\pp{\int_0^T\int_{\mathbb{R}^d}\eta(t,\xi)X(t,\xi) d\xi dt}-\mathbb{E}\pp{\int_0^T\int_{\mathbb{R}^d}\eta(t,\xi) U(t,\xi)\phi_N(\xi)d\xi dt}\\&\geq\mathbb{E}\pp{\int_0^T\int_{\mathbb{R}^d}\pr{j(X(t,\xi))-j\pr{U(t,\xi)}}\phi_N(\xi)d\xi dt}\\
&\geq  \mathbb{E}\pp{\int_0^T\int_{\mathbb{R}^d}\Psi(U(t,\xi))\pr{X(t,\xi)-U(t,\xi)}\phi_N(\xi)d\xi dt},\end{split}\end{equation}
where we have, once again, used the convexity of $j$ and the sub-gradient property $\Psi(u)\in \partial j(u)$. The right-hand term makes sense if $\Psi(U)$ belongs to $\mathbb{L}^2\pr{\Omega\times\pp{0,T};\mathbb{L}^{2,loc}\pr{\mathbb{R}^d}}$.

For every $N$, $\pr{X+\eta}\phi_N$ provides a regular process which is in $\mathbb{L}^2\pr{\Omega\times (0,T)\times B_{N+1};\mathbb{R}}$.  Using the monotonicity of $\Psi$ on $\mathbb{R}$, we get an $\mathbb{L}^2$-monotone realization of this operator. Hence, 
we get the existence of $Z_N$ as a unique $\mathbb{L}^2\pr{\Omega\times \pr{0,T}\times B_{N+1};\mathbb{R}}$-solution to 
\begin{equation*}
Z+\Psi(Z)=\pr{X+\eta}\phi_N.\footnote{The reader is invited to note that $Z_N=J_1\pr{\pr{X+\eta}\phi_N}$ where $J_1(r)=:x$ solves $x+\ln x=r$ for all $x>0$.}
\end{equation*}

Note that for all $N\leq M$, $Z_N$ and $Z_M$ coincide on $\Omega\times \pr{0,T}\times B_N$ ($\mathbb{P}\times dt\times Leb$-a.s.) and $Z_N\leq Z_M$. 
Secondly, the reader is invited to note that the right-hand term in the equation i.e.  $\pr{X+\eta}\phi_N$ is finite almost surely (on $\Omega\times \pr{0,T}\times B_N$), and, as a consequence, the solution $Z_N$ belongs to the domain of $\Psi$, or, equivalently, $Z_N>0$,  on $\Omega\times \pr{0,T}\times B_N$ ($\mathbb{P}\times dt\times Leb$-a.s.).\\ 

For every $N\geq 1$, we define
\begin{equation*}
U_N:=Z_N\mathbf{1}_{B_N}+X\mathbf{1}_{\mathbb{R}^d\setminus B_N}
\end{equation*}
and $U:= \sup_{N} Z_N$. 
One easily sees that $U=Z_N>0$, $\mathbb{P}\times dt\times Leb$-a.s. when restricted on $\Omega\times \pr{0,T}\times B_N$. 
It is now clear that $\Psi\pr{U}$ satisfies the local integrability properties (on the relevant set $B_N$, $\Psi(Z_N)=\pr{X+\eta}\phi_N-Z_N$) and we are able to apply \eqref{Estim_3} with $U$.

The reader is invited to note that, on $B_{N+1}$,  $U+\Psi(U)=X+\eta$ in an a.s. way. 

Using this, the fact that $\phi_N$ is null outside $B_{N+1}$ and by rearranging \eqref{Estim_3}, we get 
for every $M\ge 1$, $M \le N+1$,
\begin{align*}
\mathbb{E}\pp{\int_0^T\int_{\mathbb{R}^d}\pr{X-U}^2(t)(\xi)\phi_N(\xi)d\xi dt}\leq \mathbb{E}\pp{\int_0^T\int_{\mathbb{R}^d}\pr{\eta X}(t)(\xi)\pr{1-\phi_N(\xi)}d\xi dt}.
\end{align*}
The left-hand term is non-decreasing in $N$. It follows that, for every $M\geq 1$, \begin{align*}
\mathbb{E}\pp{\int_0^T\int_{\mathbb{R}^d}\pr{X-U}^2(t)(\xi)\phi_M(\xi)d\xi dt}\leq \lim_{N\rightarrow\infty }\ \mathbb{E}\pp{\int_0^T\int_{\mathbb{R}^d}\pr{\eta X}(t)(\xi)\pr{1-\phi_N(\xi)}d\xi dt}=0,\end{align*} the equality being a consequence of the integrability of $\eta X$. It follows that $X=U$ and, thus, $\eta=\Psi(X)$, first $\mathbb{P}\times dt\times d\xi$-a.s. on $\Omega\times\pr{0,T}\times B_M$, then on the whole space.
By our previous argument, it follows that $X=(U=)Z_N>0$, $\mathbb{P}\times dt\times Leb$-a.s. on $\Omega\times \pr{0,T}\times B_N$ in a first step, then by allowing $N\rightarrow\infty$,  one gets
$X>0$ a.s. on $\Omega\times \pr{0,T}\times \mathbb{R}^d$.

\textbf{Step 8.} Uniqueness of the solution is a standard consequence of the monotonicity of the logarithm and we omit it.

\section{Appendix}
\subsection{Proof of Well-Posedness of Equation \ref{Eq_lambda_nu}}\label{A1}
\begin{proof}[Proof of Well-Posedness of Equation \ref{Eq_lambda_nu}]
We check the main assumptions in \cite[Page 56]{Prevot-Rockner}
\begin{enumerate}
\item \textit{Hemicontinuity} cf.  \cite[Page 56, (H1)]{Prevot-Rockner}\\
The fact that $\theta\mapsto\scal{A(u+\theta v),x}_{\pr{V^*,V}}$ is continuous for every $u,v,x\in V$ follows as in\cite[Page 71, (H1)]{Prevot-Rockner}.  Indeed,  owing to \eqref{V*V} combined with \eqref{EstimStrat3}, 
\begin{align*}
&\scal{A(u+\theta v),x}_{\pr{V^*,V}}=\scal{\pr{\bigtriangleup-\nu\mathbb{I}}\tilde\Psi_\lambda(u+\theta v),x}_{\pr{V^*,V}}+\frac{1}{2}\scal{\pr{\sigma\otimes\sigma}\pr{u+\theta v},x}_{\pr{V^*,V}}\\=&-\scal{\tilde\Psi_\lambda(u+\theta v),x}_2+\frac{1}{2}\pr{\scal{\pr{\sigma\otimes\sigma}\pr{u},x}_{\pr{V^*,V}}+\theta\scal{\pr{\sigma\otimes\sigma}\pr{v},x}_{\pr{V^*,V}}}.
\end{align*}The continuity of the second term is obvious. For the first term, one uses the linear growth of $\tilde\Psi_\lambda$ i.e.  $\tilde\Psi_\lambda(r)\leq c_\lambda \abs{r}$ (with $c_\lambda=\lambda+\frac{1}{\lambda}$), its continuity and concludes due to Lebesgue's dominated convergence.
\item \textit{Weak monotonicity} cf.  \cite[Page 56, (H2)]{Prevot-Rockner}\\
For $u,v\in V$,  one recalls \eqref{V*V}, followed by \eqref{SigSig} and \eqref{normSigmaH} (written for $\sigma(u)-\sigma(v)$ replacing $\sigma(x)$) to get
\begin{align*}
&\scal{A(u)-A(v),u-v}_{\pr{V^*,V}}+\frac{1}{2}\norm{\sigma(u)-\sigma(v)}^2_{\mathcal{L}_2\pr{Q^{\frac{1}{2}}\mathcal{H}^{-1};H^{-1}_\nu}}\\
=&\scal{\pr{\bigtriangleup-\nu\mathbb{I}}\pr{\tilde\Psi_\lambda(u)-\tilde\Psi_\lambda(v)},u-v}_{\pr{V^*,V}}\\&+\frac{1}{2}\scal{\pr{\sigma\otimes\sigma}(u-v),u-v}_{\pr{V^*,V}}+\frac{1}{2}\norm{\sigma(u)-\sigma(v)}^2_{\mathcal{L}_2\pr{Q^{\frac{1}{2}}\mathcal{H}^{-1};H^{-1}_\nu}}\\
\leq &\scal{\pr{\bigtriangleup-\nu\mathbb{I}}\pr{\tilde\Psi_\lambda(u)-\tilde\Psi_\lambda(v)},u-v}_{\pr{V^*,V}}+\frac{1}{2}\norm{\sigma\otimes\sigma}_{\mathcal{L}\pr{H_\nu^{-1};H_{\nu}^{-1}}}\norm{u-v}_{H^{-1}_\nu}^2\\&+\frac{1}{2}\norm{\sigma(u)-\sigma(v)}^2_{\mathcal{L}_2\pr{Q^{\frac{1}{2}}\mathcal{H}^{-1};H^{-1}_\nu}}\\
\leq&-\scal{\tilde\Psi_\lambda(u)-\tilde\Psi_\lambda(v),u-v}_{2}+C\norm{u-v}^2_{H^{-1}_\nu}\leq C\norm{u-v}^2_{H^{-1}_\nu},
\end{align*}where the last inequality follows from the monotonicity of $r\mapsto \Psi_\lambda(r)$, while $C$ is determined by $C_0$ (cf. \eqref{PropSigSig}, \eqref{normSigmaH} and \eqref{Const}).
\item \textit{Coercivity} cf.  \cite[Page 56, (H3)]{Prevot-Rockner}\\
Similar to the previous computations, one gets, for every $u\in V=\mathbb{L}^2\pr{\mathbb{R}^d}$,
\begin{align*}
&\scal{A(u),u}_{\pr{V^*,V}}+\frac{1}{2}\norm{\sigma(u)}^2_{\mathcal{L}_2\pr{Q^{\frac{1}{2}}\mathcal{H}^{-1};H^{-1}_\nu}}\\
=&-\scal{\tilde\Psi_\lambda(u),u}_{2}+C\norm{u}^2_{H^{-1}_\nu}
\leq -\lambda\norm{u}_2^2+C\norm{u}^2_{H^{-1}_\nu}.
\end{align*}
For the last inequality, one recalls that $\tilde\Psi_\lambda(r)=\Psi_\lambda(r)-\Psi_\lambda(0)+\lambda r$ and uses the monotonicity of $\Psi_\lambda$.
\item \textit{Boundedness} cf.  \cite[Page 56, (H4)]{Prevot-Rockner}\\
For every $u\in V$, one has\[\norm{Au}_{V^*}=\norm{\pr{\bigtriangleup-\nu\mathbb{I}}\tilde\Psi_\lambda(u)+\frac{1}{2}\pr{\sigma\otimes\sigma}(u)}_{V^*}\leq \norm{\tilde\Psi_\lambda u}_2+\frac{1}{2}\norm{\pr{\sigma\otimes\sigma}(u)}_{V^*}.\]
On the other hand, a simple trick (and \eqref{normSigSig}) gives
\begin{align*}
\norm{\pr{\sigma\otimes\sigma}u}_{V^*}=&\sup_{v\in V;\ \norm{v}_V\leq 1}\scal{v,\pr{\sigma\otimes\sigma}(u)}_{\pr{V,V^*}}=\sup_{v\in V;\ \norm{v}_V\leq 1}\scal{\pr{\sigma\otimes\sigma}(u),v}_{H^{-1}_\nu}\\
\leq &\norm{\sigma\otimes\sigma}_{\mathcal{L}\pr{H_\nu^{-1};H^{-1}_\nu}}\sup_{v\in V;\ \norm{v}_V\leq 1}\norm{u}_{H^{-1}_\nu}\norm{v}_{H^{-1}_\nu}\\
\leq &\norm{\sigma\otimes\sigma}_{\mathcal{L}\pr{H_\nu^{-1};H^{-1}_\nu}}\times C^2_{\mathbb{L}^2\subset H_\nu^{-1}}\sup_{v\in V;\ \norm{v}_V\leq 1}\norm{u}_{V}\norm{v}_{V}.
\end{align*}
As a consequence,  and owing again to \eqref{PropSigSig},
$\norm{Au}_{V^*}\leq (C+c_\lambda)\norm{u}_{2}$, where $c_\lambda$ is, again, the Lipschitz constant for $r\mapsto \Psi_\lambda(r)-\Psi_\lambda(0)+\lambda r$.
\end{enumerate}
As a consequence, the assumptions of \cite[Theorem 4.2.4]{Prevot-Rockner} are satisfied.
\end{proof}

\subsection{Proof of Proposition \ref{PropAux}}

\begin{proof}[Proof of Proposition \ref{PropAux}]
To this purpose, we recall that $r=J_\lambda(r)+\lambda\ln J_\lambda(r)\leq J_\lambda(r)+\lambda\pr{J_\lambda(r)-1}$ leading to $J_\lambda(r)\geq \frac{r+\lambda}{1+\lambda},\ \forall r>0$.  
\begin{itemize}
\item The first case is when $r\geq 3$.  In this case, $\Psi_\lambda(r)=\frac{r-J_\lambda(r)}{\lambda}\leq \frac{r}{\lambda+1}$ and $\frac{r}{\lambda+J_\lambda(r)}< \frac{r}{\lambda+1}$ and the conclusion follows.
\item Let us now turn to the case when $r<3$. We note that, $J_\lambda(r)\leq r$ if $r\geq 1$ and $J_\lambda(r)<1$ if $r< 1$ leading to $J_\lambda(r)\leq e^{r}, \ \forall r>0$.
We now consider (for $r>0$) the function \[\pp{\frac{r+\lambda}{1+\lambda},e^r}\ni x\mapsto f(x):=\ln x+\frac{r}{\lambda+x}.\]The derivative $f'(x)=\frac{1}{x}-\frac{r}{\pr{\lambda+x}^2}\geq 0$. Indeed, this is equivalent to proving that $x^2+\pr{2\lambda-r}x+\lambda^2\geq 0$ on the interval specified. The discriminant is $r^2-4\lambda r$. The conclusion is obvious if $r\leq 4\lambda$. 
Let us assume that $3>r>4\lambda$.  We claim that $\frac{r+\lambda}{1+\lambda}\geq \frac{r-2\lambda+\sqrt{r^2-4\lambda r}}{2}$. This is equivalent to proving that $r(1-\lambda)+2\lambda(2+\lambda)\geq (1+\lambda)\sqrt{r^2-4\lambda r}$ or, again, by taking squares and re-arranging the terms, with \[-4\lambda r^2+4\lambda \pr{3+\lambda}r+4\lambda^2\pr{2+\lambda}^2\geq 0,\] obvious for $r\leq 3$.\\
It follows that $f(J_\lambda(r))\leq f\pr{e^r}=r+\frac{r}{\lambda+e^r}<2r$ and our proof is complete.
\end{itemize}
\end{proof}

\subsection{Stratonovich on $\mathcal{H}^{-1}$}\label{AppStrat}
The elements $e_k\in\mathcal{H}^{-1}$ give an orthonormal basis $e_j\otimes e_k\in \mathcal{H}^{-1}\otimes \mathcal{H}^{-1}$.  One then defines the $\pr{\sigma\otimes\sigma}$ a trace class operator on $\mathcal{H}^{-1}\otimes \mathcal{H}^{-1}$ given by \[\pr{\sigma\otimes\sigma}\pr{e_j\otimes e_k}=\mu_k\delta_{j,k}\pr{e_j\otimes e_k}.\] Via Riesz' representation theorem (for bilinear forms),  $\sigma\otimes\sigma\in\mathcal{L}\pr{\mathcal{H}^{-1};\mathcal{H}^{-1}}$. Since $\mathcal{H}^{-1}\subset H^{-1}$, it follows that \begin{align}\label{EstimStrat1}
\norm{\sigma\otimes\sigma}_{\mathcal{L}\pr{\mathcal{H}^{-1}; H^{-1}}}\leq C_{\mathcal{H}^{-1}\subset H^{-1}}\norm{\sigma\otimes\sigma}_{\mathcal{L}\pr{\mathcal{H}^{-1};\mathcal{H}^{-1}}}.
\end{align}

Please note that the natural continuous embeddings  $\mathcal{H}^{-1}\subset H_\nu^{-1}\subset H_{\nu'}^{-1}\subset H^{-1}$, for $0<\nu\leq \nu'\leq 1$ yield 
\begin{align}\label{EstimStrat1'}
\norm{\sigma\otimes\sigma}_{\mathcal{L}\pr{\mathcal{H}^{-1}; H^{-1}_\nu}}\leq C_{\mathcal{H}^{-1}\subset H^{-1}}\norm{\sigma\otimes\sigma}_{\mathcal{L}\pr{\mathcal{H}^{-1};\mathcal{H}^{-1}}}.
\end{align}
Again, due to the dense embedding $\mathcal{H}^{-1}\subset H^{-1}$, $\sigma\otimes\sigma$ extends into a bounded linear operator on $H^{-1}$ whose norm preserves that of $\sigma\otimes\sigma\in\mathcal{L}\pr{\mathcal{H}^{-1};H^{-1}}$ and, owing to \eqref{EstimStrat1} and \eqref{EstimStrat1'},
\begin{align}\label{EstimStrat2}
\norm{\sigma\otimes\sigma}_{\mathcal{L}\pr{H^{-1}; H^{-1}_\nu}}\leq C_{\mathcal{H}^{-1}\subset H^{-1}}\norm{\sigma\otimes\sigma}_{\mathcal{L}\pr{\mathcal{H}^{-1};\mathcal{H}^{-1}}}.
\end{align}
Finally, since $\norm{\cdot}_{H_{\nu}^{-1}}\geq \norm{\cdot}_{H^{-1}}$,  for all $0<\nu\leq 1$, it follows that 
\begin{align}\label{EstimStrat2'}
\norm{\sigma\otimes\sigma}_{\mathcal{L}\pr{H^{-1}_\nu; H^{-1}_\nu}}\leq C_{\mathcal{H}^{-1}\subset H^{-1}}\norm{\sigma\otimes\sigma}_{\mathcal{L}\pr{\mathcal{H}^{-1};\mathcal{H}^{-1}}}.
\end{align}
Finally,  recalling that $\norm{\cdot}_{\mathbb{L}^2}\geq \norm{\cdot}_{H^{-1}}$ and by considering the Gelfand triple based on the inclusion $V:=\mathbb{L}^2\pr{\mathbb{R}^d}\subset H^{-1}\pr{\mathbb{R}^d}\subset V^*$,  $\sigma\otimes\sigma\in\mathcal{L}\pr{\mathbb{L}^2;V^*}$ and
\begin{align}\label{EstimStrat3}
\norm{\sigma\otimes\sigma}_{\mathcal{L}\pr{\mathbb{L}^2;V^*}}\leq C_{\mathcal{H}^{-1}\subset H^{-1}}\norm{\sigma\otimes\sigma}_{\mathcal{L}\pr{\mathcal{H}^{-1};\mathcal{H}^{-1}}}.
\end{align}

\section*{Acknowledgements}

\noindent I.C. was partially supported by the Normandie Regional Council (via the M2SiNum project) and by the French ANR grants ANR-18-CE46-0013 QUTE-HPC and COSS: ANR-22-CE40-0010-01.\\
R.F.  acknowledges support from JSPS KAKENHI Grant Numbers JP19KK0066, JP20K03669. \\
D.G. acknowledges support from the National Key R and D Program of China (NO. 2018YFA0703900), the NSF of P.R.China (NOs. 12031009, 11871037), NSFC-RS (No. 11661130148; NA150344), 111 Project (No. B12023).

\bibliographystyle{plain}      
\bibliography{bibbli}  
\end{document}